\documentclass[twoside,11pt]{article}

\usepackage{graphicx, subfigure}
\usepackage[top=1in,bottom=1in,left=1in,right=1in]{geometry}
\usepackage[sort&compress]{natbib} 
\usepackage{amsfonts, amsmath, amssymb, amsthm}
\usepackage{bbm}
\usepackage{graphicx}
\usepackage{caption}
\usepackage{subfigure}

\usepackage{color}
\definecolor{darkred}{RGB}{100,0,0}
\definecolor{darkgreen}{RGB}{0,100,0}
\definecolor{darkblue}{RGB}{0,0,150}

\usepackage{hyperref}
\hypersetup{colorlinks=true, linkcolor=darkred, citecolor=darkgreen, urlcolor=darkblue}
\usepackage{url}
\urldef\eacurl\url{http://www.math.ucsd.edu/~eariasca}

\newtheorem{prp}{Proposition}
\newtheorem{lem}{Lemma}
\newtheorem{cor}{Corollary}

\def\beq{\begin{equation}}
\def\eeq{\end{equation}}
\def\beqn{\begin{eqnarray*}}
\def\eeqn{\end{eqnarray*}}
\def\bitem{\begin{itemize}}
\def\eitem{\end{itemize}}
\def\benum{\begin{enumerate}}
\def\eenum{\end{enumerate}}
\def\bmult{\begin{multline*}}
\def\emult{\end{multline*}}
\def\bcenter{\begin{center}}
\def\ecenter{\end{center}}

\newcommand{\prpref}[1]{Proposition~\ref{prp:#1}}

\newcommand{\lemref}[1]{Lemma~\ref{lem:#1}}
\newcommand{\secref}[1]{Section~\ref{sec:#1}}
\newcommand{\figref}[1]{Figure~\ref{fig:#1}}

\def\cN{\mathcal{N}}

\def\cX{\mathcal{X}}

\def\cZ{\mathcal{Z}}

\def\bbN{\mathbb{N}}

\def\bbR{\mathbb{R}}

\def\bbZ{\mathbb{Z}}

\newcommand{\E}{\operatorname{\mathbb{E}}}
\renewcommand{\P}{\operatorname{\mathbb{P}}}
\newcommand{\Var}{\operatorname{Var}}

\def\iid{\stackrel{\rm iid}{\sim}}
\def\Bin{\text{Bin}}

\def\Unif{\text{Unif}}

\def\Pois{\text{Pois}}

\def\eps{\varepsilon}

\newcommand{\IND}[1]{\mathbbm{1}_{\{ #1 \}}}

\begin{document}

\title{The Sparse Poisson Means Model}
\author{ 
Ery Arias-Castro and Meng Wang 
}
\date{Department of Mathematics, University of California, San Diego}
\maketitle

\begin{abstract}

We consider the problem of detecting a sparse Poisson mixture.
Our results parallel those for the detection of a sparse normal mixture, pioneered by \cite{MR1456646} and \cite{dj04}, when the Poisson means are larger than logarithmic in the sample size.  
In particular, a form of higher criticism achieves the detection boundary in the whole sparse regime.  
When the Poisson means are smaller than logarithmic in the sample size, a different regime arises in which simple multiple testing with Bonferroni correction is enough in the sparse regime.  
We present some numerical experiments that confirm our theoretical findings.

\medskip


\noindent {\em Keywords:} Sparse Poisson means model, goodness-of-fit tests, multiple testing, Bonferroni's method, Fisher's method, Pearson's chi-squared test, Tukey's higher criticism, sparse normal means model.
\end{abstract}

\section{Introduction}
\label{sec:intro}
The Poisson distribution is well suited to model count data in a broad variety of scientific and engineering fields.  In this paper, we consider a stylized detection problem where we observe $n$ independent Poisson counts $X_1, \dots, X_n$ from a mixture
\beq \label{model}
X_i \sim (1 - \eps)~ \Pois(\lambda_i) + \frac{\eps}{2}~ \Pois(\lambda_i') + \frac{\eps}{2} ~\Pois(\lambda_i''),
\eeq
where 
\beq \label{param_general}
\lambda_i' = \lambda_i + \Delta_i, \quad \lambda_i'' = \max(0, \lambda_i - \Delta_i), \quad \text{for some } \Delta_i >0,
\eeq
and $\eps \in [0, 1] $ is the fraction of the non-null effects.  
All the parameters are allowed to change with $n$.
We are interested in detecting whether there are any non-null effects in the sample. Specifically, we know the null means , $\lambda_1, \dots, \lambda_n$, and our goal is to test
\beq \label{problem}
H_0: \eps=0 \quad \text{versus} \quad H_1: \eps > 0.
\eeq
Put differently, we want to address the following multiple hypotheses problem
\[
H_{0,i}: X_i \sim \Pois(\lambda_i) \quad \text{versus} \quad H_{1,i}: X_i \sim (1-\eps) \Pois(\lambda_i) + \frac{\eps}2 \Pois(\lambda_i') +\frac{\eps}2 \Pois(\lambda_i'').
\]
We do assume that $\eps$ is the same for all $i$, although this is done for ease of exposition.  

This model may arise in goodness-of-fit testing for homogeneity in a Poisson process.
Suppose we record the arrival time of alpha particles over a time period and we are interested in testing for uniformity. 
One way to do so is to partition the time period into non-overlapping intervals, and count how many particles arrived with each interval.  These counts can be modeled by a Poisson distribution. 
For this problem, and any other discrete goodness-of-fit testing problems, one would typically use Pearson's chi-squared test, but we show that, under some mild conditions, this test is (grossly) suboptimal in the sparse regime where $\eps = \eps_n = o(1/\sqrt{n})$.

In another situation, we might be interested in detecting genes that are differentially expressed.  \cite{marioni2008rna} find that the variation of count data across technical replicates can be captured using a Poisson model when the over- (or under-) dispersion is not significant. Suppose we know the Poisson mean count for each gene expressed under normal conditions and want to detect a difference in expression under some other (treatment) condition.

In the model we consider here \eqref{model} the sparsity assumption is on the number of nonzero effects, which on average is $n \eps$.  
We assume that $\eps \to 0$, so the number of nonzero effects is negligible compared to the number $n$ of bins or genes being tested.
And so there are some nonzero effects under the alternative, we assume throughout the paper that 
\beq \label{epsilon}
n \eps \to \infty.
\eeq
We note that sparsity here has a different meaning from the use in the literature on sparse multinomials \citep{MR0314193, MR0370871}.
We note that sparsity here has a different meaning from the use in the literature on sparse multinomials \cite{MR0314193, MR0370871}, where the number of the bins is large so that some bins have small expected counts.

The Poisson sparse mixture model we consider here is analogous to the normal sparse mixture model pioneered by \cite{MR1456646} and \cite{dj04}, where the normal location family $\cN(\lambda, \lambda)$ plays the role of the Poisson family ${\rm Pois}(\lambda)$.
(We note that in the normal model, one can work with $\cN(\mu, 1)$, $\mu = \sqrt{\lambda}$, without loss of generality, while such a reduction does not apply to the Poisson model.)
Our results for the Poisson model are completely parallel to those for the normal model when the Poisson means are large enough that the normalized counts
\beq \label{Z}
Z_i := (X_i - \lambda_i)/\sqrt{\lambda_i}
\eeq
are uniformly well-approximated by the standard normal distribution under the null.   
Specifically, we show that this is the case when
\beq \label{normal_cond}
\min_i \lambda_i \gg \log n.
\eeq
(For two sequences $(a_n), (b_n) \subset \bbR_+$, $a_n \gg b_n$ means that $a_n/b_n\to \infty$.)  
In particular, we show that multiple testing via the higher criticism, which \cite{dj04} developed based on an idea of J.~Tukey, is asymptotically optimal to first order, just as in the normal model.
To show this, we use care in approximating the tails of the Poisson distribution with the tails of the normal distribution.
This is done by standard moderate deviations bounds.

When the Poisson means are smaller, by which we mean
\beq \label{small}
\max_i \lambda_i \ll \log n,
\eeq
we uncover a different regime where multiple testing via Bonferroni correction is optimal in the sparse regime.  In this regime, the normal approximation to the Poisson distribution is not uniformly valid, and in fact not valid at all for those indices $i$ for which $\lambda_i$ remains fixed.
We use large deviations bounds to control the tails of the Poisson distribution. 

In any case, we assume that the expected counts are lower bounded by a positive constant, concretely
\beq \label{lambda-lb}
\lambda_i \ge 1, \quad \forall i = 1,\dots,n.
\eeq
This is to make the paper self-contained, and also because in practice it is common to pool together bins to make the expected counts larger than some pre-specified minimum.

The remainder of the paper is organized as follows.
In \secref{bound}, we derive information lower bounds under various conditions on the Poisson means.
In \secref{perfm}, we study the Pearson's chi-squared goodness-of-fit test and also the max test, which is closely related to multiple testing with Bonferroni correction, showing that none of them is optimal in all sparsity regimes. We then study the higher criticism and show that it is optimal in all sparsity regimes, matching the information bound to first-order. 
In \secref{sim}, we show the result of some numerical simulations to accompany our theoretical findings.
\secref{disc} is a discussion section.
The proofs are gathered in \secref{proofs}.
We then briefly touch on the one-sided setting in \secref{one-sided}.

\section{Information Bounds}
\label{sec:bound}

We are particularly interested in regimes where the proportion of non-null effects tends to zero as the sample size grows to infinity, i.e.~$\eps \to 0$ as $n \to \infty$. 
We follow the literature on the normal sparse mixture model \citep{MR1456646,dj04,cai2010optimal}.  
We parameterize 
\beq \label{eps}
\eps = n^{-\beta}, \quad \text{ where $\beta \in (0, 1)$ is fixed}
\eeq
and consider two regimes where the detection problem behaves quite differently: the sparse regime where $\beta \in (1/2 , 1)$ and the dense regime where $\beta \in (0, 1/2)$.
We then parameterize the Poisson means in \eqref{model} differently in each regime.
When the $\lambda_i$'s are relatively large, we are guided by the correspondence between the normal model and the Poisson model via the normalized counts \eqref{Z}.

Suppose we know the fraction $\eps$ and all null and non-null Poisson rates. By the Neyman-Pearson fundamental lemma, the most powerful test for this simple versus simple hypothesis testing problem is the likelihood ratio test (LRT). 
Hence the performance of the LRT gives an information bound for this detection problem.
We investigate this information bound by finding the conditions such that the risk (the sum of probabilities of type I and type II errors) of LRT goes to one as $n \to \infty$. We say a test is asymptotically powerful when its risk tends to zero and asymptotically powerless when its risk tends to one.
All the limits are with respect to $n \to \infty$.

\subsection{Dense Regime}
\label{sec:dense}

Guided by the correspondence with the normal model, in the dense regime where $\beta < 1/2$, we parameterize the effects as follows
\beq \label{parameter-dense}
\Delta_i = n^s \cdot \sqrt{\lambda_i},
\eeq
where $s \in \bbR$ is fixed.
Define
\beq \label{lower_den}
\rho_{\rm dense}(\beta) = \frac{\beta}2 - \frac14.
\eeq

\begin{prp} \label{prp:lower_den}
Consider the testing problem \eqref{problem} with parameterizations \eqref{eps} with $\beta < 1/2$ and \eqref{parameter-dense}.
All tests are asymptotically powerless if 
\beq \label{cond-lower-den}
s < \rho_{\rm dense}(\beta).
\eeq
\end{prp}

The expert will recognize the perfect correspondence with the detection boundary for the dense regime in the two-sided detection problem in the normal model.

\subsection{Sparse Regime}

Guided by the correspondence with the normal model, in the sparse regime where $\beta > 1/2$, we start by parameterizing the effects as follows
\beq \label{parameter-sparse}
\Delta_i = \sqrt{2 r \log n} \cdot \sqrt{\lambda_i} ,
\eeq
where $r \in (0, 1)$ is fixed.
Define
\beq \label{lower}
\rho_{\rm sparse}(\beta) = \begin{cases} 
\beta - 1/2, & 1/2 < \beta \leq 3/4, \\ 
(1 - \sqrt{1 - \beta})^2, & 3/4 < \beta < 1.
\end{cases}
\eeq

\begin{prp} \label{prp:lower}
Consider the testing problem \eqref{problem} with parameterizations \eqref{eps} with $\beta > 1/2$ and \eqref{parameter-sparse} with \eqref{normal_cond}.  
All tests are asymptotically powerless if 
\beq \label{cond-lower}
r < \rho_{\rm sparse}(\beta).
\eeq
\end{prp}

Thus, Propositions~\ref{prp:lower_den} and~\ref{prp:lower} together show that, when \eqref{normal_cond} holds, meaning that $\min_i \lambda_i \gg \log n$, the detection boundary for the Poisson model is in perfect correspondence with the detection boundary for the normal model.  

When the null means $(\lambda_i : i=1,\dots,n)$ are smaller, a different detection boundary emerges in the sparse regime.  To better describe the detection boundary that follows, we adopt the following parameterization
\beq \label{lambda_small}
\lambda_i' = \lambda_i^{1-\gamma} (\log n)^\gamma, \quad \lambda_i'' = 0, \quad \text{where $\gamma> 0$ is fixed.}
\eeq  
Indeed, this particular case corresponds to $\Delta_i = \lambda_i^{1-\gamma} (\log n)^\gamma$, and assuming the $\lambda_i$'s are smaller than $\log n$ as we do, this implies that $\lambda_i'' = 0$, as it cannot be negative.

\begin{prp} \label{prp:small_lb}
Consider the testing problem \eqref{problem} with parameterizations \eqref{eps} with $\beta > 1/2$ and \eqref{lambda_small} with \eqref{small} and \eqref{lambda-lb}.  
All tests are asymptotically powerless if $\gamma < \beta$.
\end{prp}

\section{Tests}
\label{sec:perfm}

In this section we analyze some tests that are shown to achieve parts of the detection boundary.  We find that the chi-squared test achieves the detection boundary in the dense regime, the test based on the maximum normalized count (which is closely related to multiple testing with Bonferroni correction) achieves the detection boundary in the very sparse regime, while multiple testing with the higher criticism achieves the detection boundary in all regimes.

\subsection{The chi-squared test}

We start by analyzing Pearson's chi-squared test, which rejects for large values of 
\beq \label{chi-square-stats}
D = \sum_{i=1}^n \frac{(X_i - \lambda_i)^2}{\lambda_i}.
\eeq
The rationale behind using this test is two-fold.  
On the one hand, $D = \sum_i Z_i^2$ --- where the $Z_i$'s are defined in \eqref{Z} --- is the analog of the chi-squared test that plays a role in detecting a normal mean in the dense regime.  
On the other hand, this is one of the most popular approaches for goodness-of-fit testing if one interprets $X_1, \dots, X_n$ as the counts in a sample of size $N \sim {\rm Pois}(\sum_i \lambda_i)$ with values in $\{1, \dots, n\}$. 

Although we could state a more general result, we opt for simplicity and state a performance bound when the expected counts are not too small.

\begin{prp} \label{prp:chi-squared}
Consider the testing problem \eqref{problem} with \eqref{lambda-lb}, and let $a_i = \Delta_i^2/\lambda_i$.  
Then chi-squared test is asymptotically powerful if 
\beq \label{chi2-1}
\eps \sum_i a_i \gg \sqrt{n} \quad \text{and} \quad \eps \Big(\sum_i a_i\Big)^2 \gg \sum_i a_i^2,
\eeq 
and asymptotically powerless if 
\beq \label{chi2-2}
\eps \sum_i a_i \ll \sqrt{n} \quad \text{and} \quad
\eps \sum_i a_i^2 = o(n) \quad \text{and} \quad \eps \sum_i a_i^4 = o(n^2).
\eeq 
\end{prp}

From this, we immediately obtain the following result, which at once states that the chi-squared test achieves the detection boundary in the dense regime, and does not achieve the detection boundary in the sparse regime.
\begin{cor} \label{cor:chi2}
Consider the testing problem \eqref{problem} with the lower bound \eqref{lambda-lb}.
In the dense regime, where $\beta < 1/2$ in \eqref{eps} and under the parameterization \eqref{parameter-dense}, the  chi-squared test is asymptotically powerful when $s > \rho_{\rm dense}(\beta)$ defined in \eqref{lower_den}.
In the sparse regime, where $\beta > 1/2$ in \eqref{eps} and under the parameterization \eqref{parameter-sparse}, the chi-squared test is asymptotically powerless when $r$ is constant.
\end{cor}

Other classical goodness-of-tests include the (generalized) likelihood ratio $G^2$ test and the Freeman-Tukey test.
Adapted to our context, the likelihood ratio $G^2$ test rejects for large values of 
\beq \label{likelihood-stats}
G^2 = 2 \sum_{i=1}^n X_i \log \bigg( \frac{X_i}{\lambda_i} \bigg),
\eeq
while the Freeman-Tukey test rejects large values of 
\beq \label{hellinger-stats}
H^2 = 4 \sum_{i=1}^n (\sqrt{X_i} - \sqrt{\lambda_i})^2.
\eeq
We did not investigate these tests in detail, but partial work suggests that they are (as expected) equivalent to the chi-squared in the regimes we are most interested in.

\subsection{The max test}
\label{sec:max}

In analogy with the normal model, we consider the max test which rejects large values of 
\beq \label{max}
M = \max_{i=1, \dots, n} |Z_i|,
\eeq
where the $Z_i$'s are defined in \eqref{Z}.

\begin{prp} \label{prp:max-sparse}
Consider the testing problem \eqref{problem}, parameterized by \eqref{eps} and \eqref{parameter-sparse} with \eqref{normal_cond}.
When $r > (1  - \sqrt{1 - \beta})^2$, the max test is asymptotically powerful.
\end{prp}

Hence, the max test achieves the detection boundary \eqref{lower} in the very sparse regime where $\beta \in (3/4, 1)$.  
We speculate that, just as in the normal model, the max test does not achieve the detection boundary when $\beta < 3/4$.

\subsection{The higher criticism test}

In the normal model, \cite{dj04} advocate a test based on the normalized empirical process of the $Z_i$'s.  In our case, these variables are not identically distributed.  It would make sense to convert these to P-values, then, and we will comment on that in \secref{multiple-testing}.
For now, we opt for the following definition
\beq \label{hc}
T^\star = \sup_{z \in \cZ_n} T(z), \quad T(z) := \frac{\sum_{i} \big(\IND{|Z_i| >z} - K_{\lambda_i}(z)\big)}{\sqrt{\sum_i K_{\lambda_i}(z) (1- K_{\lambda_i}(z))}},
\eeq
where 
\[
K_\lambda(z) := \P\big(|\Upsilon_\lambda - \lambda|/\sqrt{\lambda} > z\big), \quad \cZ_n = \big\{z \in \bbN : \textstyle\sum_i K_{\lambda_i}(z) (1- K_{\lambda_i}(z)) \ge \log n\big\}.
\]
We consider the higher criticism test rejects for large values of $T^\star$.
This definition extends the higher criticism of \cite{dj04}, in particular the variant ${\rm HC}+$, to the case where the test statistics are not identically distributed under the null --- and cannot be transformed to be so.
The discretization of the supremum makes the control under the null particularly simple.

\begin{prp} \label{prp:hc-sparse}
Consider the testing problem \eqref{problem}, parameterized by \eqref{eps} and \eqref{parameter-sparse} with \eqref{normal_cond}.
When $r > \rho_{\rm sparse}(\beta)$, the higher criticism test is asymptotically powerful.
\end{prp}

We speculate that, just as in the normal model, the higher criticism is also able to achieve the detection boundary in the dense regime.

\subsection{Multiple testing: Fisher, Bonferroni and Tukey}
\label{sec:multiple-testing}

We now take a multiple testing perspective.  
In multiple testing jargon, our null hypothesis $H_0$ is the {\em complete null}, since 
\[
H_0 = \bigcap_{i=1}^n H_{0, i}.
\]
Several possible definitions for P-values are possible here.  We define the P-value for the $i$th hypothesis testing problem as follows
\beq \label{pval}
p_i = G_{\lambda_i}(X_i), \quad \text{where} \quad 
G_\lambda(x) := \P(|\Upsilon_\lambda -\lambda| \ge |x - \lambda|).
\eeq
There does not seem to be a consensus on the definition of P-value for asymmetric discrete null distributions \citep{dunne1996two}.  
We speculate that any reasonable definition leads to the same asymptotic results in our context.
We note that the $p_i$'s are independent, but they are discrete, and therefore not uniformly distributed in $(0,1)$ under the complete null. 
In fact, they are not even identically distributed unless the $\lambda_i$'s are all equal.  That said, for each $i$, the null distribution of $p_i$ stochastically dominates the uniform distribution.

\begin{lem} \label{lem:pval}
\citep[Lem 3.3.1]{TSH}
For any $\lambda > 0$, 
\[
\P(G_{\lambda}(\Upsilon_\lambda) \le u) \le u, \quad \forall u \in (0,1).
\]
\end{lem}

With P-values now defined, we can draw from the literature on multiple comparisons and make correspondences with the tests that we studied in the previous sections.
 
\subsubsection*{Fisher's method}
The chi-squared test is, in our context, intimately related to multiple testing with Fisher's method, which rejects the complete null for large values of 
\beq \label{fisher}
- 2\sum_{i=1}^n \log p_i.
\eeq
We speculate that, like Pearson's chi-squared test, Fisher's method achieves the detection boundary in the dense regime.  We were able to prove it in the simpler one-sided setting.  Details are postponed to \secref{one-sided}.

\subsubsection*{Bonferroni's method}
The max test is, in turn, intimately related to multiple testing with Bonferroni's method, which rejects the (complete) null for small values of 
\[
\min_{i=1, \dots, n} p_i.
\]
In fact, the two procedures are identical when the $\lambda_i$'s are all equal.
One can show that \prpref{max-sparse} applies to the Bonferroni test also.  Instead of formally proving this, we focus on complementing the lower bound established in \prpref{small_lb}.

\begin{prp} \label{prp:bonferroni}
Consider the testing problem \eqref{problem} with parameterizations \eqref{eps} with $\beta > 1/2$ and \eqref{lambda_small} with \eqref{small}.  When $\gamma > \beta$, the Bonferroni test is asymptotically powerful.
\end{prp}

We note that the same is true if we merely focus on the large $Z_i$'s, meaning, if we replace the two-sided P-values $p_i$ with 
\beq \label{pval1}
p^{\rm one}_i = G^{\rm one}_{\lambda_i}(X_i), \quad \text{where} \quad G^{\rm one}_{\lambda}(x) := \P(\Upsilon_\lambda \ge x).
\eeq
In fact, one cannot exploit the assumption that $\lambda_i'' = 0$ for all $i$.  Indeed, if we consider the test that rejects for large values of $Y := \# \{i : X_i = 0\}$, it is asymptotically powerless.  This follows from an application of \lemref{basic}.  By a simple application of Lyapunov's central limit theorem and \eqref{lambda-lb}, $Y$ is asymptotically normal both under the null and the alternative.  Moreover,  
\[\E_0(Y) = \sum_i e^{-\lambda_i}, \quad \Var_0(Y) = \sum_i e^{-\lambda_i} (1 - e^{-\lambda_i}) \ge (1-e^{-1}) n e^{-\max_i \lambda_i} = n^{1 + o(1)},\] 
where we used \eqref{lambda-lb} and \eqref{small},
while 
\[\E_1(Y) = \sum_i \Big((1-\eps) e^{-\lambda_i} + \frac\eps2 e^{-\lambda_i'} + \frac\eps2\Big) \le (1-\eps/2) \E_0(Y) + n \eps/2 \le \E_0(Y) + n^{1-\beta},\] 
and, after some simple calculations using \eqref{lambda-lb}, 
\[\Var_0(Y) \le \Var_1(Y) \le 
(1-\eps/2)^2 \Var_0(Y) + n \eps/2 \le \Var_0(Y) + n^{1-\beta}.\]
We can easily check that the conditions of \lemref{basic} are satisfied when $\beta > 1/2$.

\subsubsection*{Tukey's higher criticism}
This brings us back to the higher criticism, which is some sense is an intermediate method between Fisher's and Bonferroni's methods.  
\cite{dj04} attribute to Tukey the idea of testing the complete null based on the maximum of the normalized empirical process of the P-values, which equivalently leads to rejecting for larges values of
\beq \label{hc0}
\max_{1 \le i \le n/2} \frac{\sqrt{n} \, (i/n - p_{(i)})} {\sqrt{p_{(i)} ( 1 -  p_{(i)})}},
\eeq
where $p_{(1)} \le \cdots \le p_{(n)}$ are the sorted P-values.  In our context where the P-values are close to, but not exactly uniformly distributed, we can show that the test based on \eqref{hc0} achieves the detection boundary when all the $\lambda_i$'s are equal.  
(Details are omitted.)
When this is not so, we are not able to conclude that this is still the case.

\section{Simulations} \label{sec:sim}

We present the result of some numerical experiments whose purpose is to see the behavior of the various tests in finite samples.  So the asymptotic analysis is relevant, we chose to work with $n = 10^4$ and $n = 10^6$.  In some bioinformatics/genetics applications, $n$ could be in the millions.  We compare the tests in terms of their power when the level is controlled at $\alpha = 0.05$ by simulation.  (We generate the test statistic 500 times under the null and take the $(1 - \alpha)$-quantile as the critical value.)  The power against a particular alternative is then obtained empirically from 200 repeats.

We note that, for the higher criticism, we work with the P-values defined in \eqref{pval} and their corresponding null distribution $F_i (t) := \P(G_{\lambda_i}(\Upsilon_{\lambda_i}) \le t)$, that is,
\beq \label{hc-sim}
{\rm HC} = \max_{t \in T} \frac{\sum_{i=1}^n (\IND{p_i \le t} - F_i(t)) }{\sqrt{\sum_{i=1}^n F_i(t) (1 - F_i(t))}},
\eeq 
where $T := \{t \in (0,1) : 1/n \le F_i(t) \le 1/2,  i=1, \dots, n\}$.
We note that \eqref{hc-sim} is a generalized form of Tukey's higher criticism \eqref{hc0} for the case where $p_i$'s are not identically distributed.  Thus we find \eqref{hc-sim} more natural than \eqref{hc}, but the two are very closely related and the latter is more easily amenable to mathematical analysis.  
In practice, we estimate $F_i$ by simulation.

\subsection{In the dense regime}

In the dense regime, we have \eqref{eps} with $\beta \in (0, 1/2)$ and the parameterization \eqref{param_general} with \eqref{parameter-dense}.

In the first set of experiments, we investigate how the test performance matches the theoretical information boundary \eqref{lower_den}.  We set $n = 10^6$, all the $\lambda_i$'s equal to $\lambda_0 = 15 > \log(n)\approx 14$, and vary $\beta$ in the range of $(0, 0.5)$ with 0.025 increments and $s$ in the range of $[-0.5, 0]$ with 0.025 increments. 
When the $\lambda_i$'s are all equal, Bonferroni's method is equivalent to the max test, and is therefore omitted.  
The results are summarized in \figref{dense_fixed_lambda}. 
We see that the phase transition phenomenon is clear.
We can see the performance of the chi-squared test and Fisher's method are similar and comparable with the higher criticism, and achieve the asymptotic detection boundary. 
As expected, the max test has hardly any power in the dense regime.
We note that very similar trends are observed in the normal means model.

In the second set of experiments, we generate settings where the $\lambda_i$'s are different.  We take $n=10^4$ and fix $\beta=0.2$, and the $\lambda_i$'s are generated iid from $\lambda_0 + {\rm Exp}(\lambda_0)$, where ${\rm Exp}(\lambda)$ denotes the exponential distribution with mean $\lambda$, and we let $\lambda_0 \in \{1, 10, 100 \}$.
The results are summarized in \figref{dense_fixed_beta}.   
We can see the chi-squared test and Fisher's method perform similarly and are the best, closely followed by the higher criticism.  The max test and the Bonferroni's method perform similarly and poorly, as expected.
The effect of $\lambda_0$ does not seem important.

\begin{figure}[h!]
\centering
\includegraphics[scale=.6]{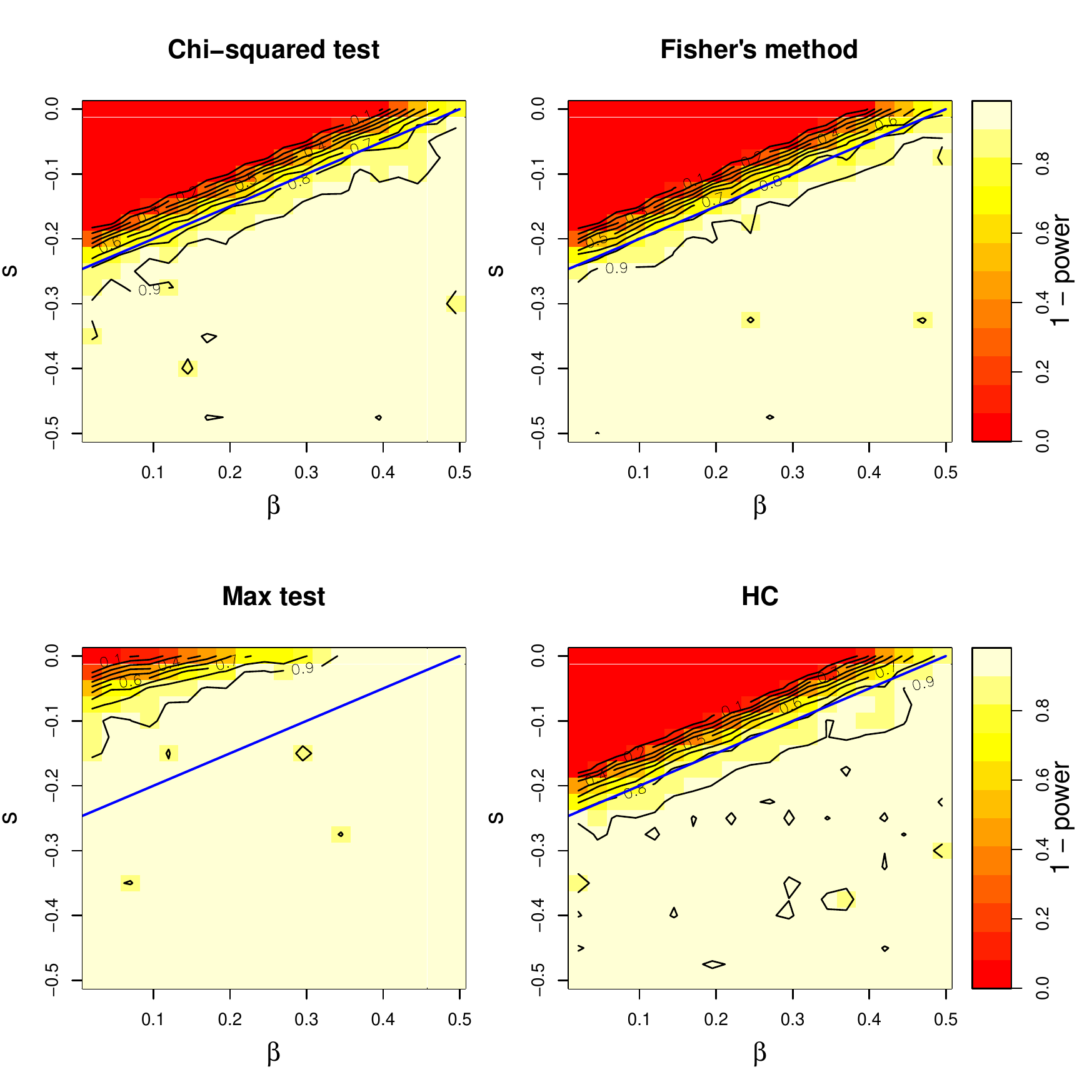}
\vspace*{-.2in}
\caption{Simulation results in the dense regime, with $n = 10^6$ and all $\lambda_i$'s equal to $\lambda_0= 15$.  The blue line is the information boundary \eqref{lower_den}.}
\label{fig:dense_fixed_lambda}
\end{figure}

\begin{figure}[h!]
\centering
\includegraphics[scale=.55]{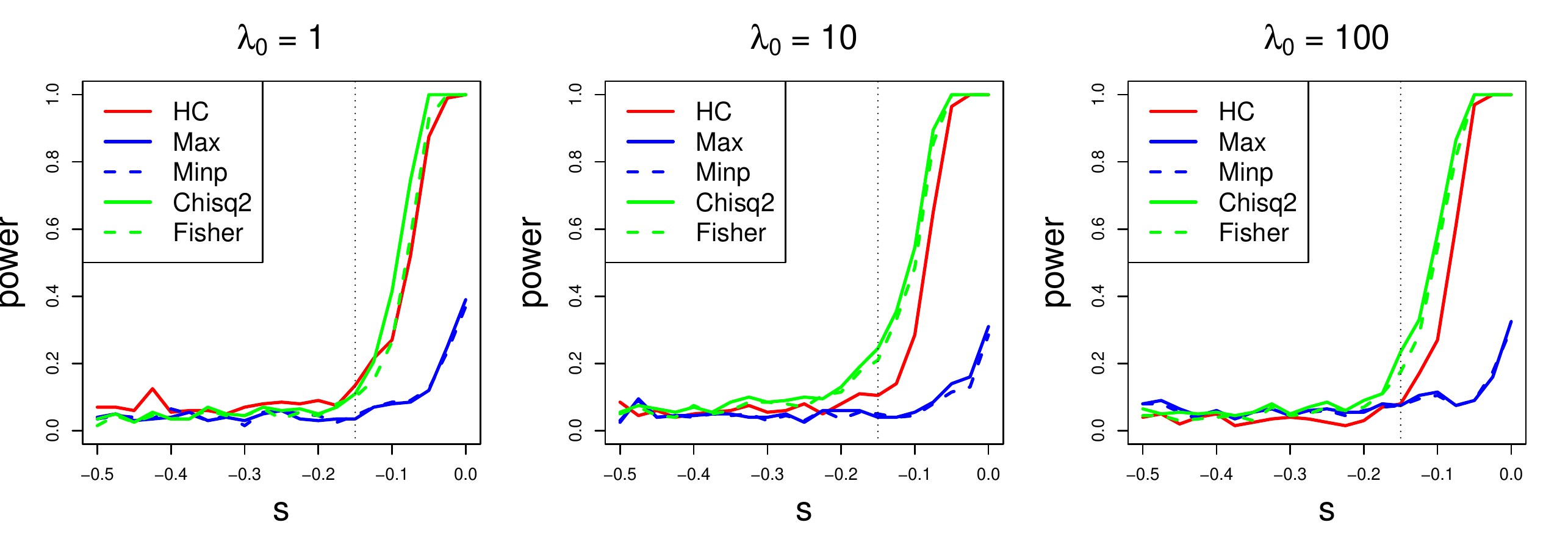}
\vspace*{-.2in}
\caption{Simulation results in the dense regime, with $n = 10^4$, $\beta = 0.2$, and the $\lambda_i$'s generated iid from $\lambda_0 + {\rm Exp}(\lambda_0)$.  The vertical dotted line is the detection threshold.
}
\label{fig:dense_fixed_beta}
\end{figure}

\newpage

\subsection{In the sparse regime}
In the sparse regime, we have \eqref{eps} with $\beta \in (1/2, 1)$ and the parameterization \eqref{param_general} with \eqref{parameter-sparse}.  The experiments are otherwise parallel to those performed in the dense regime.

In the first set of experiments, we set $n = 10^6$, means all equal to $\lambda_0 = 15$, and vary $\beta$ in the range $[0.5, 1]$ with increments of 0.025, and $r$ in the range $[0, 1]$ with increments of 0.05.
The results are summarized in \figref{sparse_fixed_lambda}.
While the chi-squared test is not competitive, as expected, we can see that the higher criticism has more power in the moderately sparse regime where $\beta \in (0.5, 0.75)$, while the max test is clearly the best in the very sparse regime where $\beta \in (0.75, 1)$.  The asymptotic detection boundary is seen to be fairly accurate, although less so as $\beta$ approaches 1, where the asymptotics take longer to come into effect.  (For example, when $n = 10^6$ and $\beta = 0.9$, there are only $n^{1-0.9} \approx 4$ anomalies.)
We note that very similar trends are observed in the normal means model.

In the second set of experiments, we set $n = 10^4$ and $\beta = 0.6$ (moderately sparse) or $\beta = 0.8$ (very sparse), and the $\lambda_i$'s are generated iid from $\lambda_0 + {\rm Exp}(\lambda_0)$, where $\lambda_0 \in \{1, 10, 100 \}$. 
The simulation results are reported in \figref{sparse_fixed_beta1} and \figref{sparse_fixed_beta3}.
We can see that the max test and Bonferroni's method perform similarly, and dominate in the very sparse regime.  
The chi-squared test is somewhat better than Fisher's method, and in some measure competitive in the moderately sparse regime, but essentially powerless in the very sparse regime.  
The higher criticism is the clear winner in the moderately sparse regime, as expected, and holds its own in the very sparse regime, although clearly inferior to the max test.
Comparing the results for different $\lambda_0$, we may conclude that, in the sparse regime, smaller counts (i.e., small $\lambda_0$) make the problem more difficult --- at least in this finite sample setting.

\begin{figure}[h!]
\centering
\includegraphics[scale=.65]{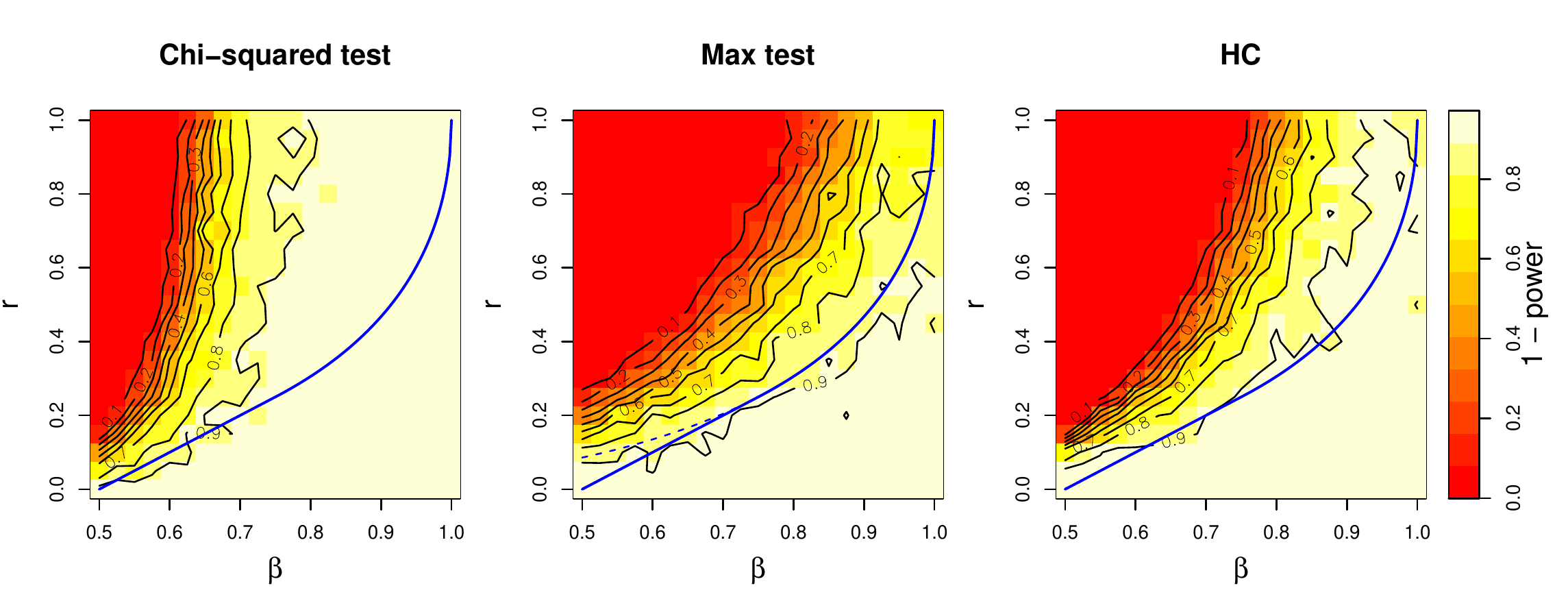}
\caption{
Simulation results in the sparse regime, with $n = 10^6$ and all $\lambda_i$'s equal to $\lambda_0= 15$.  The blue line is the information boundary \eqref{lower}. The dashed blue curve for the max test is the boundary that it can achieve.}
\label{fig:sparse_fixed_lambda}
\end{figure}

\begin{figure}[h!]
\centering
\includegraphics[scale=.55]{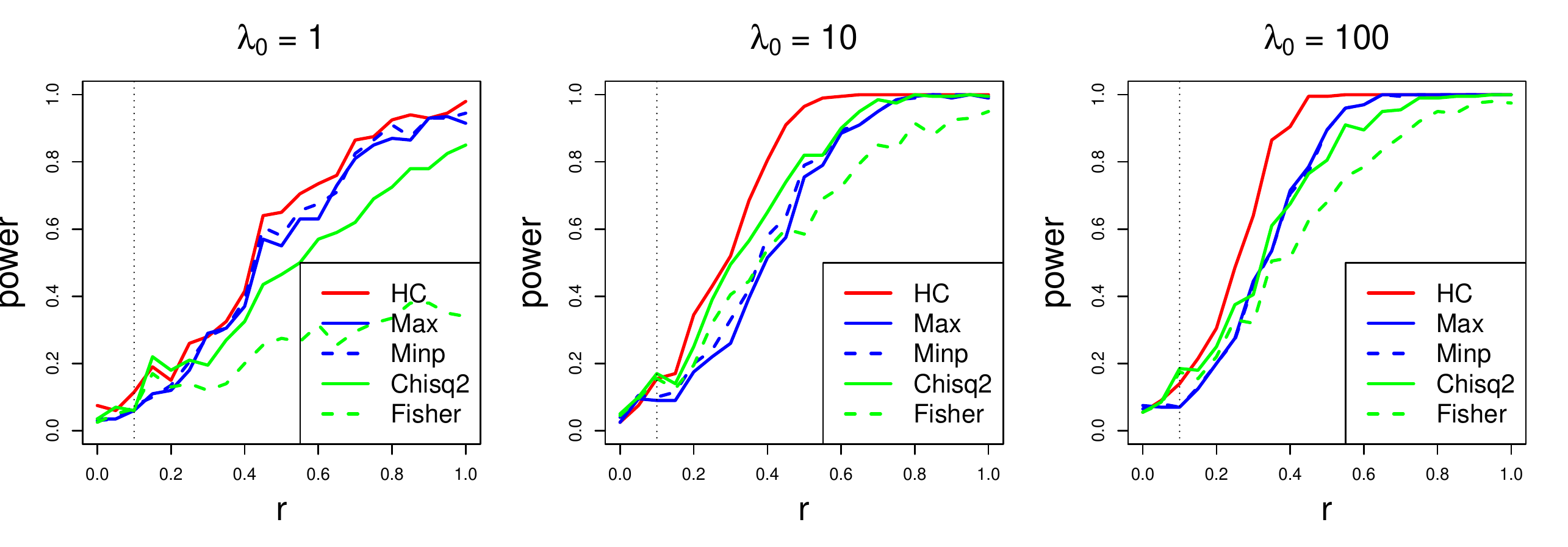}
\vspace*{-0.2in}
\caption{
Simulation results in the moderately sparse regime, with $n = 10^4$, $\beta = 0.6$, and the $\lambda_i$'s generated iid from $\lambda_0 + {\rm Exp}(\lambda_0)$.  The vertical dotted line is the detection threshold.}
\label{fig:sparse_fixed_beta1}
\end{figure}

\begin{figure}[h!]
\centering
\includegraphics[scale=.55]{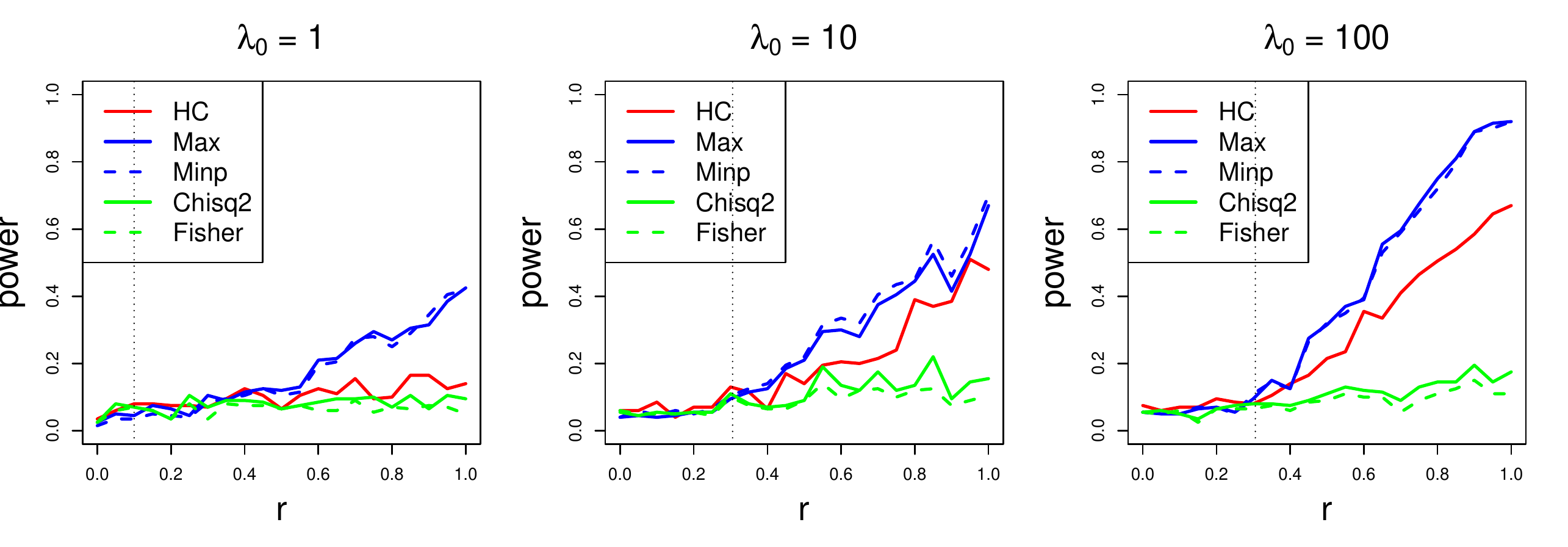}
\caption{ 
Simulation results in the very sparse regime, with $n = 10^4$, $\beta = 0.8$, and the $\lambda_i$'s generated iid from $\lambda_0 + {\rm Exp}(\lambda_0)$.  The vertical dotted line is the detection threshold.
}
\label{fig:sparse_fixed_beta3}
\end{figure}

\section{Proofs}
\label{sec:proofs}

For $a,b \in \bbR$, let $a \wedge b = \min(a,b)$ and $a \vee b = \max(a,b)$.
For two sequences of reals $(a_n)$ and $(b_n)$: $a_n \sim b_n$ when $a_n/b_n \to 1$; $a_n = o(b_n)$ when $a_n/b_n \to 0$; $a_n = O(b_n)$ when $a_n/b_n$ is bounded; $a_n \asymp b_n$ when $a_n = O(b_n)$ and $b_n = O(a_n)$; $a_n \ll b_n$ when $a_n = o(b_n)$.  Finally, $a_n \approx b_n$ when $|a_n/b_n| \vee |b_n/a_n| = O(\log n)^w$ for some $w \in \bbR$.
We use similar notation with a superscript $P$ when the sequences $(a_n)$ and $(b_n)$ are random.  In particular, $a_n = O_P(b_n)$ means that $a_n/b_n$ is bounded in probability, i.e., $\sup_{n} \P(|a_n/b_n| > x) \to 0$ as $x \to \infty$, and $a_n = o_P(b_n)$ means that $a_n/b_n \to 0$ in probability.

When $X$ and $Y$ are random variables, $X \sim Y$ means they have the same distribution.  
For a random variable $X$ and distribution $F$, $X \sim F$ means that $X$ has distribution $F$.
For a sequence of random variables $(X_n)$ and a distribution $F$, $X_n \rightharpoonup F$ means that $X_n$ converges in distribution to $F$.
Everywhere, we identify a distribution and its cumulative distribution function.
For a distribution $F$, $\bar{F}(x) = 1 - F(x)$ will denote its survival function.
We say that an event $E_n$ hold with high probability (w.h.p.) if $\P(E_n) \to 1$ as $n \to \infty$.

We let $\P_0, \E_0, \Var_0$ (resp.~$\P_{0,i}, \E_{0,i}, \Var_{0,i}$) and $\P_1, \E_1, \Var_1$ (resp.~$\P_{1,i}, \E_{1,i}, \Var_{1,i}$) denote the probability, expectation and variance under the null (resp.~null at observation $i$) and alternative (resp.~alternative at observation $i$), respectively. 
Recall that $\Upsilon_\lambda$ denotes a random variable with the Poisson distribution with mean $\lambda$, denoted $P_{\lambda}$, so that for a set $A$, $P_\lambda(A) = P(\Upsilon_\lambda \in A)$.

\subsection{Preliminaries}

We state here a few results that will be used later on in the proofs of the main results stated earlier in the paper.
We start with a couple of facts about the Poisson distribution.

The following are moderate deviation bounds for the Poisson distribution ${\rm Pois}(\lambda)$ as $\lambda \to \infty$.
\begin{lem} \label{lem:tail-bound}
Let $a : (0, \infty) \to (0, \infty)$ be such that $a(\lambda) \to \infty$ and $a(\lambda)/\lambda \to 0$ as $\lambda \to \infty$.  Then
\[
\lim_{\lambda \to \infty} \frac1{a(\lambda)} \log \P\Big(\Upsilon_\lambda \ge \lambda + \sqrt{\lambda a(\lambda)}\Big) = - \frac12
\]
and
\[
\lim_{\lambda \to \infty}  \frac1{a(\lambda)} \log \P\Big(\Upsilon_\lambda \le \lambda - \sqrt{\lambda a(\lambda)}\Big)  = - \frac12.
\]
\end{lem}

\begin{proof}
We focus on the first statement.  Let $m = [\lambda]$ and take $Y_1, \dots, Y_{m+1}$ iid Poisson with mean 1.  Fixing $\eps \in (0,1)$, we have
\[
\P\Big(\Upsilon_\lambda \ge \lambda + \sqrt{\lambda a(\lambda)}\Big) 
\le \P\Big( \sum_{i=1}^{m} Y_i + Y_{m+1} \ge m + \sqrt{m a(\lambda)}\Big)
\le {\rm I} + {\rm II},
\] 
where
\[
{\rm I} := \P\Big(\sum_{i=1}^m (Y_i-1) \ge (1 -\eps) \sqrt{m a(\lambda)}\Big),
\quad {\rm II} := \P\Big(Y_{m+1} \ge \eps \sqrt{m a(\lambda)}\Big),
\]
where in the first inequality we used the fact that $\Upsilon_\lambda$ is stochastically bounded from above by $\sum_{i=1}^{m+1} Y_i$, and in the second inequality we used the union bound.
By \citep[Th 3.7.1]{MR1619036},
\[
\frac1{a(\lambda)} \log {\rm I} \to - \frac{(1-\eps)^2}2, \quad m \to \infty.
\]
And using the fact that $\P(\Upsilon_1 \ge x)/\P(\Upsilon_1 = x) \to 1$ as $x \to \infty$, we have
\[
\log {\rm II} = \log \P\Big(\Upsilon_1 = [\eps \sqrt{m a(\lambda)}]\Big) + o(1) \sim - \eps \sqrt{m a(\lambda)} \log \sqrt{m a(\lambda)}, \quad m \to \infty.
\]
Since $a(\lambda) = o(m)$, we have that ${\rm II} = o({\rm I})$, and conclude that
\[
\limsup_{\lambda \to \infty} \frac1{a(\lambda)} \log \P\Big(\Upsilon_\lambda \ge \lambda + \sqrt{\lambda a(\lambda)}\Big) \le - \frac{(1-\eps)^2}2,
\]
and because $\eps > 0$ is arbitrary, we may take $\eps = 0$ in this last display.  The reverse inequality is proved similarly.
\end{proof}

The following are concentration bounds for the Poisson distribution.
For a real $x$, let $\lceil x \rceil$ (resp.~$\lfloor x \rfloor$) denote the smallest (resp.~largest) integer greater (resp.~smaller) than or equal to $x$.
\begin{lem} \label{lem:chernoff}
For $x \ge 0$, define $h(x) = x \log(x) - x + 1$, with $h(0) = 0$.  Then, for any $\lambda > 0$,
\[
- \lambda h(\lceil x \rceil/\lambda) - \tfrac12 \log \lceil x \rceil - 1 \le \log \P\Big(\Upsilon_\lambda \ge x\Big) \le - \lambda h(x/\lambda), \quad \forall x \ge \lambda,
\]
and
\[
- \lambda h(\lfloor x \rfloor/\lambda) - \tfrac12 \log \lfloor x \rfloor - 1 \le \log \P\Big(\Upsilon_\lambda \le x\Big) \le - \lambda h(x/\lambda), \quad \forall \, 0 \le x \le \lambda.
\]
\end{lem}

\begin{proof}
The upper bounds result from a straightforward application of Chernoff's bound.  For the first lower bound, take $x \ge \lambda$ and let $m = \lceil x \rceil$.  Then 
\[\log \P\Big(\Upsilon_\lambda \ge x\Big) 
\ge \log \P\Big(\Upsilon_\lambda = m \Big) = \log \Big(e^{-\lambda} \frac{\lambda^{m}}{ m !}\Big) 
\ge - \lambda h(m/\lambda) - \log m - 1 
\]
using the fact that $m! \le m^{m+1/2} e^{-m+1}$.  The second lower bound is proved similarly.
\end{proof}

The following is Berry-Esseen's theorem applied to the Poisson distribution ${\rm Pois}(\lambda)$ as $\lambda \to \infty$.
\begin{lem} \label{lem:berry-bound}
There is a universal constant $C > 0$ such that
\[
\sup_{x \in \bbR} \Bigg|\P \Bigg(\frac{\Upsilon_{\lambda} - \lambda}{\sqrt{\lambda}} \le x \Bigg) - \Phi(x)\Bigg| \le C/\sqrt{\lambda}.
\]
\end{lem}

\begin{proof}
Let $m = \lceil \lambda \rceil$ be the smallest integer greater than or equal to $\lambda$.
It is enough to prove the result when $\lambda \ge 1$, in which case $1/2 \le \lambda/m \le 1$.
Take $Y_1, \dots, Y_m$ are iid $\Pois(\lambda/m)$, so that $\Upsilon_\lambda \sim \sum_{i=1}^m Y_i$. We have $\E(Y_i) = \Var(Y_i) = \lambda/m$ and $\E (|Y_i - \lambda/m|^3) \le \E(\Upsilon_1^3) < \infty$. The result now follows by the Berry-Esseen theorem.
\end{proof}

The following lemma is standard, and appears for example in \citep{arias2013distribution}.

\begin{lem}  \label{lem:basic}
Consider a test that rejects for large values of a statistic $T_n$ with finite second moment, both under the null and alternative hypotheses.  Then the test that rejects when $T_n \ge t_n := \E_0(T_n) + \frac{a_n}2 \sqrt{\Var_0(T_n)}$ is asymptotically powerful if
\beq \label{basic1}
a_n := \frac{\E_1(T_n) - \E_0(T_n)}{\sqrt{\Var_1(T_n) \vee \Var_0(T_n)}} \to \infty.
\eeq 
Assume in addition that $T_n$ is asymptotically normal, both under the null and alternative hypotheses.  Then the test is asymptotically powerless if  
\beq \label{basic2}
\frac{\E_1(T_n) - \E_0(T_n)}{\sqrt{\Var_0(T_n)}} \to 0 \quad \text{ and } \quad \frac{\Var_1(T_n)}{\Var_0(T_n)} \to 1.
\eeq 
\end{lem}

Finally, we state without proof the following simple result.
\begin{lem} \label{lem:rho-ineq}
The function $f(\beta) = (1 - \sqrt{1 - \beta})^2 - (\beta - 1/2)$ is nonnegative and strictly increasing on $(3/4 ,1)$.
\end{lem}

\subsection{Proof of \prpref{lower_den}}
Here we use the second moment method without truncation, which amounts to proving that $\Var_0(L) \to 0$, or equivalently, $\E_0(L^2) \le 1 +o(1)$, where $L$ is the likelihood ratio
\[
L = \prod_{i=1}^n L_i,
\]
where
\beq \label{Li}
L_i := \frac{(1 - \eps) P_{\lambda_i}(X_i) + \frac \eps 2 P_{\lambda_i'}(X_i) + \frac \eps 2 P_{\lambda_i''}(X_i) }{P_{\lambda_i}(X_i)}.
\eeq
We have $\E_0(L^2) = \prod_{i=1}^n \E_0(L_i^2)$, where
\begin{align*}
\E_0(L_i^2) &= \sum_{x=0}^{\infty} \frac{\big[(1 - \eps) P_{\lambda_i}(x) + \frac \eps 2 P_{\lambda_i'}(x) + \frac \eps 2 P_{\lambda_i''}(x)\big]^2}{P_{\lambda_i}(x)} \\
&= \sum_{x = 0}^{\infty} \frac{\big[ (1-\eps) e^{-\lambda_i} \frac{\lambda_i^{x}}{x !} +  \frac{\eps}{2} e^{-\lambda_i'} \frac{\lambda_i'^{x}}{x !} + \frac{\eps}{2} e^{-\lambda_i''} \frac{\lambda_i''^{x}}{x !} \big]^2}{ e^{-\lambda_i} \frac{\lambda_i^{x}}{x !}}  \\
&= (1-\eps)^2 + 2(1-\eps)\eps + \frac{\eps^2}{4} e^{-2 \lambda_i' + \lambda_i + \frac{\lambda_i'^2}{\lambda_i}} + \frac{\eps^2}{4} e^{-2 \lambda_i'' + \lambda_i + \frac{\lambda_i''^2}{\lambda_i}}  + \frac{\eps^2}{2} e^{-\lambda_i' - \lambda_i'' + \lambda_i + \frac{\lambda_i' \lambda_i''}{\lambda_i}} \\
&= 1 + \frac{\eps^2}2 \big[ (e^{n^{2 s}} - 1) + (e^{-n^{2s}} - 1) \big] \\
&= 1 + a_n, \quad \text{where } a_n := \eps^2 \big[ \cosh(n^{2 s}) - 1 \big].
\end{align*}
In the third line we used the fact that $\sum_{x=0}^\infty \lambda^x/x! = e^\lambda$ for all $\lambda \in \bbR$, and in the fourth line we used \eqref{parameter-dense}. 
Condition \eqref{cond-lower-den} and the fact that $\beta < 1/2$ imply that $s < 0$, and a Taylor expansion gives 
$
a_n \le n^{-2\beta +4 s},
$
eventually.
We deduce that $\E_0(L^2) \le (1+a_n)^n$, and the RHS tends to 1 when $n a_n \to 0$, which is the case because of \eqref{cond-lower-den}.

\subsection{Proof of \prpref{lower}}
We use the truncated second moment method of Ingster in the form put forth by \cite{butucea2013detection}.
Define 
\[
x_i = \lambda_i + \sqrt{2 (1+\eta) \log(n)} \sqrt{\lambda_i}, \quad y_i = \lambda_i - \sqrt{2 (1+\eta) \log(n)} \sqrt{\lambda_i},
\]
where $\eta > 0$ is chosen small enough that \eqref{eta-choice} and \eqref{eta-choice2} hold simultaneously.

Define the truncated likelihood function,
\[
\tilde{L} = \prod_{i=1}^n L_i \IND{A_i}, \quad A_i := \{ y_i \le X_i \le x_i \},
\]
where $L_i$ is defined in \eqref{Li}.
As in \cite{butucea2013detection}, it suffices to prove that
\[
\E_0 (\tilde L) \geq 1 + o(1) \quad \text{and} \quad \E_0(\tilde L^2) \leq 1 + o(1).
\]

\noindent 
{\bf First moment.}  We have
\[
\E_0 (\tilde L) = \prod_{i=1}^n \E_0 (L_i \IND{A_i}) = \prod_{i=1}^n \P_{1} (A_i),
\]
with
\[
\P_{1} (A_i^c) = (1-\eps) P_{\lambda_i}(A_i^c) + \frac\eps2 P_{\lambda'_i}(A_i^c) + \frac\eps2 P_{\lambda''_i}(A_i^c).
\]
Applying \lemref{tail-bound}, using \eqref{parameter-sparse} and the fact that $\lambda'_i \sim \lambda''_i \sim \lambda_i \gg \log n$ because of \eqref{normal_cond}, we get
\[
P_{\lambda_i}(A_i^c) \le n^{- 1-\eta +o(1)}, \quad
P_{\lambda'_i}(A_i^c) \vee P_{\lambda''_i}(A_i^c)\le n^{- (\sqrt{1+\eta} - \sqrt{r} )^2 + o(1)},
\]
uniformly over $i = 1,\dots, n$.
Hence,
\[
\P_{1} (A_i) \ge 1 - a_{n}, \quad
\text{for some $a_{n} \le n^{-1-\eta + o(1)} + \eps n^{- (\sqrt{1+\eta} - \sqrt{r})^2 + o(1)}$,}
\]
which in turn implies 
\[
\E_0(\tilde L) \ge (1 - a_n)^n.
\]
Using the expression for $\eps$, we have
\[
n a_n \le n^{-\eta +o(1)} + n^{1 -\beta - (\sqrt{1 + \eta}- \sqrt{r})^2 + o(1)}.
\]
By \eqref{cond-lower} and \lemref{rho-ineq}, for any $\beta \in (1/2, 1)$, we have $r < \rho_{\rm sparse}(\beta) \le (1 - \sqrt{1-\beta})^2 \le (\sqrt{1+\eta} - \sqrt{1-\beta})^2$, which in turn implies that $1 -\beta - (\sqrt{1 + \eta}- \sqrt{r})^2 < 0$.  Therefore, $n a_n = (1)$, and so $\E_0(\tilde L) \ge 1 - o(1)$.

\bigskip\noindent 
{\bf Second moment.}  
We have
\[
\E_0(\tilde L^2) = \prod_{i=1}^n \E_{0}(L_i^2 \IND{A_i}),
\] 
where
\begin{align}
\E_{0}(L_i^2 \IND{A_i}) 
&= \sum_{y_i \le x \le x_i}\frac{\big[(1 - \eps) P_{\lambda_i}(x) + \frac \eps 2 P_{\lambda_i'}(x) + \frac \eps 2 P_{\lambda_i''}(x)\big]^2}{P_{\lambda_i}(x)} \notag\\
&= \sum_{y_i \le x \le x_i} (1 - \eps)^2 P_{\lambda_i}(x) + \eps (1-\eps) \big(P_{\lambda_i'}(x) + P_{\lambda_i''}(x)\big) + \frac{\eps^2}4 \frac{\big(P_{\lambda_i'}(x) + P_{\lambda_i''}(x)\big)^2}{P_{\lambda_i}(x)} \notag\\
&\le (1-\eps)^2 + 2\eps (1-\eps) + \frac{\eps^2}4 \sum_{y_i \le x \le x_i} \frac{2 \big[e^{-\lambda_i'} \frac{\lambda_i'^{x}}{x !}\big]^2 + 2 \big[e^{-\lambda_i''} \frac{\lambda_i''^{x}}{x !} \big]^2}{ e^{-\lambda_i} \frac{\lambda_i^{x}}{x !}}  \notag\\
&= 1 - \eps^2 + \frac{\eps^2}2 \sum_{y_i \le x \le x_i} \frac1{x!} \bigg [e^{-2 \lambda_i' + \lambda_i} \big(\tfrac{\lambda_i'^2}{\lambda_i}\big)^{x} + e^{-2 \lambda_i'' + \lambda_i} \big(\tfrac{\lambda_i''^2}{\lambda_i}\big)^{x} \bigg] \notag \\
&\le 1 + \frac{\eps^2}2 \Big[e^{(\lambda_i' - \lambda_i)^2/\lambda_i}  P_{\lambda_i'^2/\lambda_i}([0, x_i]) + e^{(\lambda_i'' - \lambda_i)^2/\lambda_i}  P_{\lambda_i''^2/\lambda_i}([y_i, \infty)) \Big] \notag\\ 
&\le 1 + \frac12 n^{-2\beta +2r} \Big[P_{\lambda_i'^2/\lambda_i}([0, x_i]) + P_{\lambda_i''^2/\lambda_i}([y_i, \infty))\Big]. \label{truncated-2nd}
\end{align}
In the third line we used the fact that $(a+b)^2 \le 2 a^2 + 2 b^2$ for all $a,b \in \bbR$.

Let $\delta = \rho_{\rm sparse}(\beta) - r$, which is strictly positive by \eqref{cond-lower}

{\bf Case 1.}  When $\beta \le 3/4$, $-2 \beta + 2 r = -1 -\delta$, and we can bound the 2nd term in \eqref{truncated-2nd} by $n^{-1-\delta}$.

{\bf Case 2.}   When $\beta > 3/4$, we distinguish two sub-cases.  Let $f$ be the function defined in \lemref{rho-ineq}.  In the first case, $\delta \ge 1/2$, in which case $-2 \beta + 2 r = -1 - 2 [\delta - f(\beta)] < -1$ for any $\beta < 1$, so that we can bound the 2nd term in \eqref{truncated-2nd} by $n^{-1 - 2 [\delta - f(\beta)]}$.  In the second case, $\delta < 1/2$, so that $f^{-1}(\delta)$ exists in $(3/4, 1)$.  If $\beta < f^{-1}(\delta)$, then $f(\beta) < \delta$ and the same bound on the 2nd term in \eqref{truncated-2nd} applies.  If $\beta \ge f^{-1}(\delta)$, we have $r = \rho_{\rm sparse}(\beta) -\delta \ge \rho_{\rm sparse}(f^{-1}(\delta)) -\delta = f^{-1}(\delta) - 1/2 > 1/4$.  Fix $\eta > 0$ small enough that
\beq \label{eta-choice}
f^{-1}(\delta) - 1/2 > (1+\eta)/4.
\eeq
Since $\lambda'_i \sim \lambda''_i \sim \lambda_i \gg \log n$, 
\[
\lambda_i'^2/\lambda_i = \lambda_i + 2 \sqrt{2 r \log(n)} \sqrt{\lambda_i} (1 + o(1)) \quad \text{and} \quad \lambda_i''^2/\lambda_i = \lambda_i - 2 \sqrt{2 r \log(n)} \sqrt{\lambda_i} (1 + o(1)).
\]
Hence,  
\begin{align*}
P_{\lambda_i'^2/\lambda_i}([0,x_i]) 
&=  P_{\lambda_i'^2/\lambda_i}\Big(Z_i \le - (2 \sqrt{r} -\sqrt{1 + \eta}) \sqrt{2 \log(n)}(1 + o(1))\Big) \\
&= n^{- (2 \sqrt{r} - \sqrt{1 + \eta})^2 + o(1)}, 
\end{align*}
and
\begin{align*}
P_{\lambda_i''^2/\lambda_i}([y_i, \infty)) 
&=  P_{\lambda_i''^2/\lambda_i}\Big(Z_i \ge ( 2 \sqrt{r} - \sqrt{1 + \eta}) \sqrt{2 \log(n)}(1 + o(1))\Big) \\
&= n^{- (2 \sqrt{r} - \sqrt{1 + \eta})^2 + o(1)},
\end{align*}
because of \lemref{tail-bound}, and the fact that $2 \sqrt{r} > \sqrt{1+\eta}$ by our choice of $\eta$ in \eqref{eta-choice}.
We can thus bound on the 2nd term in \eqref{truncated-2nd} by
\[
n^{2r - 2\beta - (2 \sqrt{r} - \sqrt{1 + \eta})^2 + o(1)}.
\]
When $\eta = 0$, the exponent is equal to
\[
2r - 2\beta - (2 \sqrt{r} - 1)^2 = -1 - 2 (\beta - 1 + (1 - \sqrt{r})^2) < -1 - 2 (\beta - 1 + (1 - \rho_{\rm sparse}^{1/2}(\beta))^2) = -1.
\]
Hence, when $\eta > 0$ is small enough,  
\beq \label{eta-choice2}
2r - 2\beta - (2 \sqrt{r} - \sqrt{1 + \eta})^2 < -1.
\eeq

We conclude that $\E_0(\tilde{L}_i^2) \le 1 + o(n^{-1})$, uniformly in $i$, which implies that 
\[
\E_0(\tilde{L}^2) \le (1 + o(n^{-1}))^n = 1 + o(1).
\]

\subsection{Proof of \prpref{small_lb}}

The proof parallels that of \prpref{lower}.  Here we define
\[
x_i = (1+c) \frac{\log n}{\log (\zeta_i)}, \quad \zeta_i := \frac{\log n}{\lambda_i},
\]
where $c$ is a small positive constant that will be chosen later on, and consider the following truncated likelihood
\[
\tilde{L} = \prod_{i=1}^n L_i \IND{A_i}, \quad A_i := \{X_i \le x_i \}.
\]

\bigskip\noindent 
{\bf First moment.}  
Taking into account the fact that $\lambda_i'' = 0$, it suffices to prove that 
\[
P_{\lambda_i}(A_i^c) + \eps P_{\lambda'_i}(A_i^c) = o(1/n),
\]
uniformly over $i = 1,\dots,n$.
Let $h(t) = t \log t -t + 1$.  There is $t_0$ such that, for $t \ge t_0$, $h((1+c) t) \ge (1+c/2) t \log t$.
Note that $x_i/\lambda_i \ge \zeta_i/ \log(\zeta_i) \ge \zeta_{\rm min}/ \log(\zeta_{\rm min}) \to \infty$, eventually, since \eqref{small} implies $\zeta_{\rm min} := \min_i \zeta_i \to \infty$.  Hence, using \lemref{chernoff}, we get
\[
\log P_{\lambda_i}(A_i^c) \le - \lambda_i h(x_i/\lambda_i) \le - \lambda_i (1+c/2) \frac{\zeta_i}{\log (\zeta_i)} \log\Big(\frac{\zeta_i}{\log (\zeta_i)}\Big) \le - (1+c/3) \log n, 
\] 
as soon as $\zeta_{\rm min}/ \log(\zeta_{\rm min})$ is large enough.
This implies that $\max_i P_{\lambda_i}(A_i^c) = o(1/n)$.

Note that $(\log n)/\lambda_i' = \zeta_i^{1-\gamma}$.
So we also have $x_i/\lambda_i' \ge \zeta_{\rm min}^{1-\gamma}/ \log(\zeta_{\rm min}) \to \infty$ eventually, and using \lemref{chernoff}, we get
\[
\log P_{\lambda_i'}(A_i^c) \le - \lambda_i' h(x_i/\lambda_i') \le - \lambda_i' (1+c/2) \frac{\zeta_i^{1-\gamma}}{\log (\zeta_i)} \log\Big(\frac{\zeta_i^{1-\gamma}}{\log (\zeta_i)}\Big) \le - (1+c/3)(1-\gamma) \log n, 
\] 
as soon as $\zeta_{\rm min}^{1-\gamma}/ \log(\zeta_{\rm min})$ is large enough.
Since $\gamma < \beta$ by assumption, this implies $\eps \max_i P_{\lambda_i'}(A_i^c) = o(1/n)$.

\bigskip\noindent 
{\bf Second moment.}  
Taking into account the fact that $\lambda_i'' = 0$, it suffices to prove that 
\[
\eps^2 \Big[e^{(\lambda_i' - \lambda_i)^2/\lambda_i}  P_{\lambda_i'^2/\lambda_i}([0, x_i]) + e^{\lambda_i} \Big] = o(1/n),
\]
uniformly over $i = 1,\dots,n$.
We quickly see that
\[
\eps^2 e^{\lambda_i} \le n^{-2 \beta + 1/\zeta_{\rm min}} = n^{-2 \beta + o(1)} = o(1/n), 
\]
since $\beta > 1/2$ is fixed.  
For the other term, we distinguish two cases.

{\bf Case 1.}
First, assume that $\gamma < 1/2$.  
Then 
\[
\eps^2 e^{(\lambda_i' - \lambda_i)^2/\lambda_i}  P_{\lambda_i'^2/\lambda_i}([0, x_i]) \le \eps^2 e^{\lambda_i'^2/\lambda_i} \le n^{-2\beta + \zeta_{\rm min}^{2 \gamma-1}} = n^{-2 \beta + o(1)} = o(1/n).
\]

{\bf Case 2.}
Now, assume that $\gamma \ge 1/2$.  
Then $\lambda_i'^2/(\lambda_i x_i) \ge \zeta_{\rm min}^{2\gamma-1} \log \zeta_{\rm min} \to \infty$, so that applying \lemref{chernoff}, we get
\[
\log P_{\lambda_i'^2/\lambda_i}([0, x_i]) \le -\frac{\lambda_i'^2}{\lambda_i} h(x_i \lambda_i/\lambda_i'^2) = x_i \log(\lambda_i'^2/(\lambda_i x_i)) + x_i -\frac{\lambda_i'^2}{\lambda_i},
\] 
with 
\beq \label{xlog}
x_i \log(\lambda_i'^2/(\lambda_i x_i)) \le (1+c) (\log n) \Big[(2 \gamma-1) + \frac{\log \log \zeta_{\rm min}}{\log \zeta_{\rm min}} \Big],
\eeq
so that
\beqn
\eps^2 e^{(\lambda_i' - \lambda_i)^2/\lambda_i}  P_{\lambda_i'^2/\lambda_i}([0, x_i]) 
&\le& \exp\Big[ -2 \beta \log n - 2 \lambda_i' + \lambda_i + x_i \log(\lambda_i'^2/(\lambda_i x_i)) + x_i\Big] \\
&\le& n^{-2 \beta  + (1+c) (2 \gamma -1) + o(1)},
\eeqn
uniformly over $i = 1, \dots, n$, since in addition to \eqref{xlog}, we also have $- 2 \lambda_i' + \lambda_i + x_i \le x_i \le (1+c) \log n/\log \zeta_{\rm min} = o(\log n)$.
Since $\gamma < \beta$, we may choose $c > 0$ small enough that $-2 \beta  + (1+c) (2 \gamma -1) < -1$.

\subsection{Proof of \prpref{chi-squared}}

We have
\[
\E(\Upsilon_\lambda) = \lambda, \quad \Var(\Upsilon_\lambda) = \lambda, \quad \E(\Upsilon_\lambda-\lambda)^3=\lambda, \quad \E(\Upsilon_\lambda-\lambda)^4 = 3 \lambda^2 + \lambda.
\]
Using this, for the Poisson model \eqref{model}, we have
\[
\E_0(D) = n, \quad \E_1(D) = n +  \eps \sum_{i=1}^n \frac{\Delta_i^2}{\lambda_i}, \quad \Var_0(D) = 2n + \sum_{i=1}^n \frac1{\lambda_i}, 
\]
and, after some simple but tedious calculations,
\[
\Var_1(D) = \Var_0(D) + \eps R, 
\]
where 
\[
R = \sum_{i=1}^n \Bigg[\frac{4 \Delta_i^2}{\lambda_i} + \frac{7 \Delta_i^2}{\lambda_i^2} + \frac{(1-\eps) \Delta_i^4}{\lambda_i^2}\Bigg] \le C \sum_{i=1}^n (a_i + a_i^2)\ ,
\]
for some universal constant $C>0$, using \eqref{lambda-lb}.
We have $\E_1(D) - \E_0(D) = \eps \sum_{i=1}^n a_i$ and $\Var_0(D) \vee \Var_1(D) \le 2n + \sum_{i=1}^n \frac1{\lambda_i} + C \eps \sum_{i=1}^n (a_i + a_i^2)$.
Because of \eqref{lambda-lb}, we have $\sum_{i=1}^n \frac1{\lambda_i} = O(n)$ and then, by \eqref{chi2-1}, we have $\eps \sum_{i=1}^n a_i \to \infty$.  With this and the second part of \eqref{chi2-1}, it becomes straightforward to see that the first part of \lemref{basic} applies and we conclude that way.

We now prove that the chi-squared test is asymptotically powerless under \eqref{chi2-2}.  
For one thing, this condition implies that $\Var_1(D) \sim \Var_0(D)$, based on \eqref{chi2-2} and the bound on $R$ above, and also that
$\E_1(D) - \E_0(D) \ll \sqrt{\Var_1(D) \vee \Var_0(D)}$.
It therefore suffices to prove that $D$ is asymptotically normal both under the null and under the alternative.
We have $D = \sum_i Z_i^2$, where $Z_i^2 := (X_i - \lambda_i)^2/\lambda_i$, and these being independent random variables, it suffices to verify Lyapunov's conditions.   
Some straightforward calculations yield 
\[
\E_0(Z_i^2 - \E_0(Z_i^2))^4 = \E_0(Z_i^2 -1)^4 \le C \Big(1 + \frac1{\lambda_i} + \frac1{\lambda_i^2} + \frac1{\lambda_i^3}\Big),
\]
for some constant $C > 0$, and using \eqref{lambda-lb}, we get
\[
\Var_0(D)^{-2} \sum_{i=1}^n \E_0(Z_i^2 - 1)^4 = O(1/n^2) n = O(1/n) = o(1).
\]
With some more work, and using \eqref{lambda-lb}, we also obtain 
\[
\E_1(Z_i^2 - \E_1(Z_i^2))^4 \le C \big(1 + \eps (a_i + a_i^4) \big),
\]
for some constant $C > 0$, so that
\[
\Var_1(D)^{-2} \sum_{i=1}^n \E_1(Z_i^2 - \E_1(Z_i^2))^4 = O(1/n^2) \sum_{i=1}^n \big(1 + \eps (a_i + a_i^4) \big) 
= o(1),
\]
which is an immediate consequence of \eqref{chi2-2}.

\subsection{Proof of \prpref{max-sparse}}

When $r > (1 - \sqrt{1 - \beta})^2$, there exists a $\delta > 0$ such that $r > (\sqrt{1 + \delta}- \sqrt{1 - \beta})^2$.
Define the threshold $c_n = \sqrt{2(1 + \delta) \log(n)}$. Under the null, by the union bound and \lemref{tail-bound}, under \eqref{normal_cond},
\[
\P_0(M \ge c_n) \le \sum_{i=1}^n \P_0(|Z_i| \ge c_n) = n ^{- \delta + o(1)} = o(1).
\]

Under the alternative, define $I' := \{i: X_i \sim \Pois(\lambda_i') \}$ and $p_{i,n}' = \P(\Upsilon_{\lambda_i'} \ge \lambda_i + c_n \sqrt{\lambda_i})$.  By \lemref{tail-bound}, we have 
\[
p_n' := \min_{i = 1,\dots, n} p_{i,n}' \ge n^{-(\sqrt{1 + \delta} - \sqrt{r})^2 + o(1)}.
\]
We then derive the following
\beqn
\P_1(M \ge c_n) 
&\ge& \P \big(\max_{i \in I'} Z_i \ge c_n \big) \\
&=& 1 - \E\Big[ \prod_{i \in I'} (1 - p_{i,n}') \Big] \\
&\ge& 1 - \E\Big[ (1 - p_n')^{|I'|} \Big] \\
&\ge& 1 - (1 - p_n')^{n \eps/4} -o(1),
\eeqn
where in the last line we used the fact that $|I'| \sim \Bin(n, \eps/2)$, so that $|I'| \ge n\eps/4$ with probability tending to one.
Since
\[
(n \eps) p_n' \ge n ^{1 - \beta - (\sqrt{1 + \delta} - \sqrt{r})^2 + o(1)} \to \infty, \quad n \to \infty,
\]
because $r > (\sqrt{1 + \delta}- \sqrt{1 - \beta})^2$ by construction, we have $\P_1(M \ge c_n) \to 1$ as $n \to \infty$, as we needed to prove.

\subsection{Proof of \prpref{hc-sparse}}
We first control the size of the statistic $T^\star$ under the null.  For each $z \in \bbR$, the variables $\IND{|Z_i| > z}, i=1,\dots,n,$ are independent Bernoulli, with respective parameters $K_{\lambda_i}(z), i=1,\dots,n$.  
We can therefore apply Bernstein's inequality, to get
\[
\log \P_0\Big( \textstyle\sum_i (\IND{|Z_i| > z} - K_{\lambda_i}(z)) > t \sigma(z) \Big) \le - \frac{\frac12 t^2}{1 + \frac13 t/\sigma_z}, \quad \forall t \ge 0,
\] 
where $\sigma_z^2 := \sum_i K_{\lambda_i}(z) (1-K_{\lambda_i}(z))$.
Choosing $t = 2 \sqrt{\log n}$ and letting $z \in \cZ_n$, so that $\sigma_z \ge \frac12 t$, the right-hand side is bounded by $- \frac65 \log n$. 
Thus, applying the union bound, we get
\[\P_0\Big( T^\star > 2 \sqrt{\log n} \Big) \le |\cZ_n| n^{-6/5},\]
where $|\cZ_n|$ is the cardinality of $\cZ_n$.  We now show that $|\cZ_n|$ is subpolynomial in $n$.
By \lemref{chernoff}, we have 
\[
K_\lambda(z) \le e^{-\lambda h(1 + z/\sqrt{\lambda})} + e^{-\lambda h(1 - z/\sqrt{\lambda})},
\]
where $h$ is defined in that lemma, and extended as $h(t) = \infty$ when $t < 0$, so that this inequality is true for all $\lambda, z > 0$.
Note that $h(1+t) = t^2/2 + O(t^3)$ when $t = o(1)$.
Take $z_n = \sqrt{3 \log n}$.  Because of \eqref{normal_cond}, uniformly in $i=1,\dots,n$, we have $K_{\lambda_i}(z_n) \le n^{-3/2 + o(1)}$, and in particular, $\sigma_{z_n}^2 \le n^{-1/2 + o(1)} < \log n$ eventually.  Hence, by monotonicity, $z \le z_n$ for all $z \in \cZ_n$.  In particular, $|\cZ_n| \le z_n$.
Hence, we arrive at the conclusion that $\P_0\big( T^\star > 2 \sqrt{\log n} \big) = o(1)$.

Suppose we are now under the alternative.
We focus on the case where $r < 1$, which is more subtle.
Consider $z_n(q) = \lfloor \sqrt{2 q \log n} \rfloor$, defined for any $q > 0$.
By \lemref{tail-bound}, when \eqref{normal_cond} and \eqref{parameter-sparse} hold, we have $K_{\lambda_i}(z_n(q)) = n^{-q + o(1)}$ uniformly over $i$.  Hence, 
\[p_{n,i}^0(q) := \P_0(|Z_i| > z_n(q)) = K_{\lambda_i}(z_n(q)) = n^{-q + o(1)},\] 
uniformly over $i$.
In particular, when $q \in (0,1)$ is fixed, $\sigma_{z_n(q)}^2 = n^{1-q +o(1)} \ge \log n$, eventually, in which case $z_n(q) \in \cZ_n$.
Hence, for each fixed $q \in (0,1)$, we have $T^\star \ge T(z_n(q))$ for $n$ large enough, and so it suffices to prove that, for some well-chosen $q$, $\P_1(T(z_n(q)) \le 2 \sqrt{\log n}) = o(1)$.

Assume $q > r$.
By \lemref{tail-bound} again, this time under the alternative, and also assuming that \eqref{normal_cond} and \eqref{parameter-sparse} hold,
then
\beqn
K_{\lambda_i'}(z_n(q)) &=& n^{-(\sqrt{q} - \sqrt{r})^2 + o(1)}, \\
K_{\lambda_i''}(z_n(q)) &=& n^{-(\sqrt{q} - \sqrt{r})^2 + o(1)},
\eeqn
uniformly over $i = 1,\dots, n$.
Hence, 
\beqn
p_{n,i}^1(q) := \P_1(|Z_i| > z_n(q)) 
&=& (1-\eps) K_{\lambda_i}(z_n(q)) + \frac\eps2 K_{\lambda_i'}(z_n(q)) + \frac\eps2 K_{\lambda_i''}(z_n(q)) \\
&=& p_{n,i}^0(q) + n^{-\beta-(\sqrt{q} - \sqrt{r})^2 + o(1)}.
\eeqn
It follows that
\beqn
\E_1(T(z_n(q))) &=& \frac{\sum_i (p_{i, n}^1(q) - p_{i, n}^0(q)) }{\sqrt{\sum_i p_{i, n}^0(q) (1- p_{i, n}^0(q))}} 
=  \frac{ n^{1 -\beta - (\sqrt{q} - \sqrt{r})^2 + o(1)}}{\sqrt{n^{1 - q + o(1)}}} \\
&=& n^{1/2 + q/2 - \beta - (\sqrt{q} - \sqrt{r})^2 + o(1)} 
\eeqn
and
\[
\Var_1(T(z_n(q))) = \frac{\sum_{i=1}^n p_{i, n}^1(q)( 1 - p_{i, n}^1(q))}{\sum_i p_{i, n}^0(q) (1- p_{i, n}^0(q))}  
= O(1) \vee n^{q - \beta - (\sqrt{q} - \sqrt{r})^2 + o(1)}.
\]

First, assume that $r < 1/4$, so that $r - (\beta -1/2) = r - \rho_{\rm sparse}(\beta) > 0$, where the equality follows from \eqref{lower} and the fact that $r < 1/4$.
We take $q = 4 r$ and get 
\[
\E_1(T(z_n(4r))) = n^{r  - \beta + 1/2 + o(1)},
\]
with $r  - \beta + 1/2 = r  - (\beta - 1/2) > 0$, 
and 
\[
\Var_1(T(z_n(4r))) = O(1) \vee n^{-\beta + 3 r + o(1)}.
\]
By Chebyshev's inequality, we have
\beqn
\P_1(T(z_n(4r) < 2 \sqrt{\log n}) \le \frac{\Var_1(T(z_n(4r))}{(\E_1(T(z_n(4r)) - 2 \sqrt{\log n})^2} &=& \frac{O(1) \vee n^{-\beta + 3 r + o(1)}}{n^{1 + 2 r - 2\beta + o(1)}} \\
&=&  \begin{cases} 
O(n^{-1 - 2r + 2 \beta + o(1)}), & \text{if~} \beta \ge 3 r, \\ 
O(n^{\beta + r -1 + o(1)}), & \text{if~} \beta < 3r,
\end{cases}
\eeqn
with $-1 - 2r + 2 \beta < -1 - 2(\beta-1/2)+ 2 \beta = 0$ and $\beta + r -1< r + 1/2 + r -1 < 0$ since $r < 1/4$.

Now, assume that $r \ge 1/4$, which together with $r > \rho_{\rm sparse}(\beta)$ and $r \ge 1/4$ implies that $r > (1 - \sqrt{1 - \beta})^2$, which in turn forces $1 - \beta - (1 - \sqrt{r})^2 > 0$.
Take $r < q<1$ such that $1 - \beta - (\sqrt{q} - \sqrt{r})^2 > 0$
Then 
\[
 \E_1(T(z_n(q))) = n^{1 - \beta - (\sqrt{q} - \sqrt{r})^2 +o(1)}
\]
and
\[
\Var_1(T(z_n(q))) = n^{1 - \beta - (\sqrt{q} - \sqrt{ r})^2 + o(1)}.
\]
Thus, by Chebyshev's inequality,
\[
\P_1(T(z_n(q)) < 2 \sqrt{\log n}) \le \frac{\Var_1(T(z_n(q))}{(\E_1(T(z_n(q))) - 2 \sqrt{\log n})^2} = n^{(\sqrt{q} - \sqrt{r})^2 - 1 + \beta + o(1)} = o(1).
\]

\subsection{Proof of \prpref{bonferroni}}

Consider the situation under the null.
Because of \lemref{pval}, we have
\[
\min_i p_i \ge^{\rm sto} \min_i u_i, \quad u_1, \dots, u_n \iid \Unif(0,1).
\]
Therefore, under the null we have $\P_0(\min_i p_i \le \omega_n/n) = o(1)$ for any sequence $\omega_n = o(1)$.  
Take $\omega_n = 1/\log n$.

Under the alternative, let $I' = \{i : X_i \sim {\rm Pois}(\lambda_i')\}$.
Note that $\lambda_i h(X_i/\lambda_i) \ge \log(n/\omega_n)$ implies 
\[
p_i = \P(\Upsilon_{\lambda_i} \ge X_i | X_i) 
\le \omega_n/n,
\] 
where the equality is due to the fact that, necessarily, $X_i \ge 3 \lambda_i$ eventually, and the inequality comes from \lemref{chernoff}.
Thus, defining $q_i = \P\big(\lambda_i h(\Upsilon_{\lambda_i'}/\lambda_i) \ge \log (n/\omega_n)\big)$, we arrive at
\beqn
\P_1(\min_i p_i > \omega_n/n) 
&\le& \P\Big(\min_{i \in I'} p_i > \omega_n/n\Big) \\
&\le& \E \Big[\textstyle\prod_{i \in I'} (1-q_i)\Big] \\
&\le& (1-q_{\rm min})^{n \eps/4},
\eeqn
where $q_{\rm min} := \min_{i=1,\dots,n} q_i$, and in the last line we used the fact that $|I'| \sim \Bin(n, \eps/2)$, so that $|I'| \ge n\eps/4$ with probability tending to one.
Note that 
\[q_i = \P\Big(\Upsilon_{\lambda_i'} \ge b_i \Big), \quad b_i := \lambda_i h^{-1}\big(\tfrac{\log (n/\omega_n)}{\lambda_i}\big),\]
where for $t \ge 0$, $h^{-1}(t)$ is defined as the unique $x \ge 1$ such that $h(x) = t$.
Notice that $h^{-1}(t) \sim t/\log t$ when $t \to \infty$.
Let $\zeta_i = \log n/\lambda_i$, so that $\zeta_{\rm min} := \min_i \zeta_i \to \infty$ when \eqref{small} holds.
We have 
\[b_i / \lambda_i' \sim \log n / (\lambda_i' \log \zeta_i) = \zeta_i^{1-\gamma}/\log \zeta_i \ge \zeta_{\rm min}^{1-\gamma}/\log \zeta_{\rm min} \to \infty.\]
Therefore, applying the first lower bound in \lemref{chernoff}, we get 
\[
\log q_i \ge - \lambda_i' h(\lceil b_i \rceil/\lambda_i') - \tfrac12 \log \lceil b_i \rceil - 1 \sim - b_i \log (b_i/\lambda_i') \sim - \frac{\log n}{\log \zeta_i} \log( \zeta_i^{1-\gamma}) = - (1-\gamma) \log n,
\]
uniformly over $i=1,\dots,n$ because $\min_i (b_i \wedge (b_i/\lambda_i') \wedge \zeta_i) \to \infty$.  In particular, $q_{\rm min} \ge n^{\gamma-1 + o(1)}$, implying that $n \eps q_{\rm min} \ge n^{\gamma-\beta +o(1)} \to \infty$, because $\gamma > \beta$ by assumption.  We conclude that $\P_1(\min_i p_i > \omega_n/n) = o(1)$, as we needed to prove.

\section{The one-sided setting} \label{sec:one-sided}

Up until now, we considered a two-sided setting, partly motivated by the important example of goodness-of-fit testing, where Pearson's chi-squared test is omnipresent.  Simpler is a one-sided setting, where instead of \eqref{model} we have
\beq \label{model1}
X_i \sim (1 - \eps)~\Pois(\lambda_i) + \eps~\Pois(\lambda_i'),
\eeq
together with $\lambda_i' = \lambda_i + \Delta_i$ and $\eps \in [0,1]$, and address the problem \eqref{problem} in this context.  
Such a model may be relevant in some image processing applications where the goal is to detect an anomaly in the form pixels with higher-intensity.

\subsection{Dense Regime}
In the dense regime where \eqref{eps} holds with $\beta < 1/2$, we consider the same parameterization \eqref{parameter-dense}.
Define
\beq \label{lower_den1}
\rho^{\rm one}_{\rm dense}(\beta) = \beta -\frac12.
\eeq

\begin{prp} \label{prp:lower_den1}
Consider the testing problem \eqref{problem} in the one-sided setting \eqref{model1}, with parameterizations \eqref{eps} with $\beta < 1/2$ and \eqref{parameter-dense}.
All tests are asymptotically powerless if 
\beq \label{cond-lower-den1}
s < \rho^{\rm one}_{\rm dense}(\beta).
\eeq
\end{prp}

The proof is parallel to that of \prpref{lower_den} --- in fact simpler --- and is omitted.
We note that this detection boundary is in direct correspondence with that in the normal model \citep{cai2010optimal}.

In the one-sided setting, the chi-squared test does not achieve the detection boundary.  However, its one-sided version does.  Indeed, consider the test that rejects for large values of 
\beq \label{chi-square-stats1}
\sum_{i=1}^n \frac{X_i - \lambda_i}{\sqrt{\lambda_i}}.
\eeq

\begin{prp} \label{prp:chi-squared1}
Consider the testing problem \eqref{problem} in the one-sided setting \eqref{model1}, with \eqref{lambda-lb}, and let $a_i = \Delta_i/\sqrt{\lambda_i}$.  
The test based on \eqref{chi-square-stats1} is asymptotically powerful if \eqref{chi2-1} holds.  In particular, with parameterization \eqref{eps} with $\beta < 1/2$ and \eqref{parameter-dense}, the test is asymptotically powerful when $s > \rho^{\rm one}_{\rm dense}(\beta)$.
\end{prp}

The proof is parallel to, and in fact much simpler than, that of \prpref{chi-squared}, and is omitted.

All the arguments are simpler in the one-sided setting, so much so that we are able to analysis Fisher's method.
In the one-sided setting, instead of \eqref{pval}, define the P-values as in \eqref{pval1}.
Note that \lemref{pval} still applies.

\begin{prp} \label{prp:fisher}
Consider the testing problem \eqref{problem} in the one-sided setting \eqref{model1}, with \eqref{lambda-lb}, and let $a_i = \Delta_i/\sqrt{\lambda_i}$.  
Fisher's test (based on \eqref{fisher}) is asymptotically powerful if 
\[
\eps \sum_i (a_i \wedge 1) \gg \sqrt{n}.
\]
In particular, with parameterization \eqref{eps} with $\beta < 1/2$ and \eqref{parameter-dense}, Fisher's test is asymptotically powerful when $s > \rho^{\rm one}_{\rm dense}(\beta)$.
\end{prp}

To streamline the proof, which is somewhat long and technical, we implicitly focused on the most interesting case where the $a_i$'s are bounded, but this is not intrinsic to the method.  In fact, the test has increasing power with respect to each $a_i$.  
The technical proof is detailed in \secref{proof_fisher}.

\subsection{Sparse Regime}
In the sparse regime, the same results apply.  In particular, the detection boundary described in Propositions~\ref{prp:lower} and~\ref{prp:small_lb} applies.  
The max test --- now based on $\max_i Z_i$ --- and Bonferroni's method achieve the detection boundary in the very sparse regime ($\beta > 3/4$). 
The higher criticism is now based on 
\[T^\star = \sup_{x \in \cX_n} T(x), \quad T(x) := \frac{\sum_{i} \big(\IND{X_i > x} - G_{\lambda_i}(x)\big)}{\sqrt{\sum_i G_{\lambda_i}(x) (1- G_{\lambda_i}(x))}},
\]
with definition \eqref{pval1} and 
\[
\cX_n = \big\{x \in \bbN : \textstyle\sum_i G_{\lambda_i}(x) (1- G_{\lambda_i}(x)) \ge \log n\big\},
\]
and it achieves the detection boundary over the whole sparse regime ($\beta > 1/2$).  The technical arguments are parallel, and in fact simpler, and are omitted.

\subsection{Proof of \prpref{fisher}} \label{sec:proof_fisher}
\def\l{\lambda}

Let $V$ be the statistic \eqref{fisher}.
We seek to apply \lemref{basic}, which is based on the first two moments, under the null and under the alternative.  
In what follows, $\l \ge 1$ and $\l' = \l + a \sqrt{\l}$ with $0 < a \le 1$ for some constant $C > 0$.

{\bf Difference in means.}
For $\l > 0$, $g_\l(x) = \P(\Upsilon_\l = x)$, $G_\l(x) = \P(\Upsilon_\l \ge x)$, and $F_\l(X) = -2 \log G_\l(X)$.  
We have
\[
\E_\l(F_\l) = - 2 \sum_{x \ge 0} [\log G_\l(x)] g_\l(x) = 2 \sum_{x \ge 1} [\log G_\l(x-1) -\log G_\l(x)] G_\l(x),
\]
using the fact that $g_\l(x) = G_\l(x) - G_\l(x+1)$ and $G_\l(0) = 1$.  A similar expression holds for $\E_{\l'}(F_\l)$, and combined, we get
\[\begin{split}
\E_{\l'}(F_\l) -\E_\l(F_\l) 
&= 2 \sum_{x \ge 1} [\log G_\l(x-1) -\log G_\l(x)] [G_{\l'}(x) - G_\l(x)] \\
&= 2 \sum_{x \ge 1} \log \Big[1 + \frac{g_\l(x-1)}{G_\l(x)}\Big] [G_{\l'}(x) - G_\l(x)]. \\
\end{split}\]   
In that case, the summands are positive, since $\log G_\l(x-1) \ge \log G_\l(x)$ by monotonicity of $G_\l$, and $G_{\l'}(x) \ge G_\l(x)$ by the fact that $\Upsilon_{\l'}$ stochastically dominates $\Upsilon_\l$ when $\l' > \l$.
To get a lower bound, we may thus restrict the sum to any subset of $x$'s, and we choose $x \in I_\l := [\l, \l + \sqrt{\l}]$.  
Since $\l \ge 1$, $I_\l \ne \emptyset$.
Moreover,  
\[\frac1{C_0} \le G_\l(x) \le C_0, \quad \forall x \in I,\] 
for some universal constant $C_0>1$.  
This is a direct consequence of \lemref{berry-bound} when $\l \ge \l_0$ for some large-enough constant $\l_0$, and otherwise, it comes from the fact that $G_\l(x) > 0$ for all pairs $(\l,x)$ such that $\l < \l_0$ and $x \in I_\l$, which is a finite set of pairs.
We also have 
\[
\frac1{C_1 \sqrt{\l}} \le g_\l(x) \le \frac{C_1}{\sqrt{\l}}, \quad \forall x \in [\l-1, \l + \sqrt{\l}].
\]
for a numeric constant $C_1 > 1$.  
Indeed, by Stirling's formula, we have $g_\l(x) \asymp x^{-1/2} \exp(- \l h(x/\l))$, where we recall that $h(x) = x\log x - x+1$, and we have $x^{-1/2} \asymp \l^{-1/2}$, and also $\l h(x/\l)\asymp 1$, uniformly over $x \in I_\l$. 
We also have 
\[
\frac{g_\nu(x)}{g_\l(x)} \ge 1/C_2, \quad \forall x \in I, \quad \forall \nu \in [\l, \l'],
\]
for a numeric constant $C_2 > 1$.  
Indeed, 
\[\begin{split}
\frac{g_\nu(x)}{g_\l(x)} 
&\ge \exp\big[-\nu+\l + \l \log(\nu/\l)\big] = \exp\big[- \tfrac12 \tfrac{(\nu-\l)^2}{\l} + O\big(\tfrac{(\nu-\l)^3}{\l^2}\big) \big] \\
&\ge \exp\big[- \tfrac12 a^2 + O(a^3/\sqrt{\l})\big],
\end{split}\]
which is bounded from below when $a$ is bounded from above.
Using the fact that $\partial_\l G_\l(x) = g_\l(x-1)$, by the mean-value theorem, we also have $G_{\l'}(x) - G_\l(x) = (\l' - \l) g_{\l_x}(x)$, for some $\l_x \in [\l, \l']$, which together with the last two bounds implies that 
\[
G_{\l'}(x) - G_\l(x) \ge a/C_3, \quad \forall x \in I_\l,
\]
for a numeric constant $C_3 > 1$.  
Gathering all these results, we derive
\[
\E_{\l'}(F_\l) -\E_\l(F_\l) \ge 2 \sum_{x \in I_\l \cap \bbZ} \log\Big[1 + \frac1{C_0 C_1 \sqrt{\l}}\Big] \frac{a}{C_3} \ge \frac{a}{C_4},
\]
for another constant $C_4 > 1$, because $|I_\l \cap \bbZ| \asymp \sqrt{\lambda}$.

{\bf Variances.}
When $X \sim g_\l$, $G_\l(X)$ stochastically dominates $U \sim \Unif(0,1)$, and because $t \to (\log t)^2$ is decreasing on $(0,1)$, we have 
\[
\E_\l (F_\l^2) \le C_5 := 4 \E [(\log U)^2] < \infty.
\]
Let $R_{\l,\l'}(X) = g_{\l'}(X)/g_\l(X)$.  We have
\[
\E_{\l'} (F_\l^2) = \E_\l [F_\l^2 \ R_{\l,\l'}] \le 2 \E_\l (F_\l^2) + \E_\l[F_\l^2 \ R_{\l,\l'} \IND{R_{\l,\l'} > 2}].
\]
Note that $R_{\l,\l'}(x) > 2$ if, and only if, $x > x_* := (\Delta + \log 2)/\log(1 + \Delta/\l)$.  Hence,
\[
\E_\l[F_\l^2 R_{\l,\l'} \IND{R_{\l,\l'} > 2}] = \sum_{x \ge x_*} [\log G_\l(x)]^2 g_{\l'}(x).
\]

\begin{lem}[Bohman's inequality, as in Sec 35.1.8 of \cite{dasgupta2008asymptotic}] 
For any $\l > 0$, 
\[\P\big(\Upsilon_\l \ge x\big) \ge \bar{\Phi}\big(\tfrac{x-\l}{\sqrt{\l}}\big), \quad \forall x \in \bbN.\]
\end{lem}

This lemma, together with Mills ratio, yields
\[
\sum_{x \ge x_*} [\log G_\l(x)]^2 g_{\l'}(x) = O(1) \sum_{x \ge x_*} \Big(\frac{x-\l}{\sqrt{\l}}\Big)^4 x^{-1/2} \exp[- \l h(x/\l)],
\]
since, for any $x \ge x_*$, $\frac{x-\l}{\sqrt{\l}} \ge t_* := \frac{x_*-\l}{\sqrt{\l}} \asymp 1/a \ge 1$.
We learn in \citep[Prop 1, p. 441]{ShoWel} that $h(1+t) \ge \frac12 t^2 (1+\frac13 t)^{-1}$ for all $t \ge 0$.  Hence, 
\[\l h(x/\l) \ge \frac{(x-\l)^2}{2\l} \frac1{1+\frac13 \frac{x-\l}{\l}} \ge \frac{(x-\l)^2}{4\l} \IND{x \le 4\l} + \frac34 (x-\l) \IND{x > 4\l}.\]
Thus 
\[\begin{split}
\sum_{x \ge x_*} \Big(\frac{x-\l}{\sqrt{\l}}\Big)^4 x^{-1/2} \exp[- \l h(x/\l)] & \le \sum_{x_* \le x \le 4 \l} \Big(\frac{x-\l}{\sqrt{\l}}\Big)^4 x^{-1/2} \exp\Big[- \frac{(x-\l)^2}{4\l}\Big] \\
&\quad + \sum_{x > 4 \l} \Big(\frac{x-\l}{\sqrt{\l}}\Big)^4 x^{-1/2} \exp\Big[- \frac34 (x-\l)\Big].
\end{split}\]
The first sum is bounded by 
\[
\l^{-1/2} \sum_{t=t_*}^{\lceil 3\sqrt{\l} \rceil} \ \sum_{x = \lfloor \l + t \sqrt{\l}\rfloor}^{\lfloor \l + (t+1) \sqrt{\l} \rfloor} (t+1)^4 e^{-t^2/4} \le \sum_{t \ge t_*} (t+1)^4 e^{-t^2/4} = o(1).
\]
The second sum is bounded by 
\[
\l^{-5/2} \sum_{x > 4\l} (x -\l)^4 e^{-\frac34 (x -\l)} = \l^{-5/2} \sum_{x > 3\l} x^4 e^{-\frac34 x} \le C_6,
\]
for a numeric constant $C_6$, since $\l \ge 1$.
We conclude that 
\[
\E_{\l'} (F_\l^2) \le C_7,
\]
for some numeric constant $C_7$.

{\bf Conclusion.}
Since the test has increasing power with respect to each $a_i$, we may assume that $a_i \le 1$ for all $i$.
Let $F_{\l_i} = - 2 \log G_{\l_i}(X_i)$ and notice that $V = \sum_i F_{\l_i}$ is our test statistic.
We have
\[
\E_1(V) - \E_0(V) = \sum_i \big[\E_1(F_{\l_i}) - \E_0(F_{\l_i})\big] = \eps \sum_i \big[\E_{\l_i'}(F_{\l_i}) - \E_{\l_i}(F_{\l_i}) \big] \ge \eps \sum_i \frac{a_i}{C_4},
\]
and 
\[
\Var_0(V) \le \sum_i \E_{\l_i}(F_{\l_i}^2) \le n C_5,
\]
as well as
\[
\Var_1(V) \le \sum_i \E_{1}(F_{\l_i}^2) \le \sum_i \E_{\l_i'}(F_{\l_i}^2) \le n C_7.
\]
By \lemref{basic}, we conclude that the test is asymptotically powerful when
\[
\eps \sum_i a_i \gg \sqrt{n}.
\]

\section{Discussion} \label{sec:disc}

We drew a strong parallel between the Poisson means model and the normal means model.  The correspondence is in fact exact when all the $\lambda_i$'s are at least logarithmic in $n$.  When the $\lambda_i$ are smaller, we uncovered a new detection boundary in the sparse regime.  We studied the chi-squared test, the max test and the higher criticism, which are shown here to have similar properties as in the normal model.  Motivated by the higher criticism, we also advocated a multiple testing approach to Poisson means model, and studied emblematic approaches such as Fisher's and Bonferroni's methods, which are indeed shown to achieve the detection boundary in some regime/model.
An open direction might be to adapt the method of \cite{MR2275246} for estimating the number of non null effects in the Poisson means model.

\bigskip
\subsection*{Acknowledgements}
This work was partially supported by a grants from the US National Science Foundation (NSF) (DMS 1120888 and 1223137).

\end{document}